\documentclass[11pt,a4paper]{article}
\usepackage{mathrsfs}
\usepackage{multirow}
 \usepackage{epsfig}
 \usepackage{epstopdf}
\usepackage[T1]{fontenc}
\usepackage{geometry}
\usepackage{amsmath}
\usepackage{bm}
\usepackage{amsbsy,latexsym,amsfonts, epsfig, color, authblk, amssymb, graphics, bm}
\usepackage{epsf,slidesec,epic,eepic}
\usepackage{fancybox}
\usepackage{fancyhdr}
\usepackage{setspace}
\usepackage{cases}
\usepackage{nccmath}
\usepackage{url}
\usepackage[colorlinks, citecolor=blue]{hyperref}

\newtheorem{theorem}{Theorem}[section]
\newtheorem{lemma}{Lemma}[section]

\newtheorem{definition}{Definition}[section]
\newtheorem{example}{Example}[section]
\newtheorem{proposition}{Proposition}[section]
\newtheorem{corollary}{Corollary}[section]

\newtheorem{remark}{Remark}[section]

\newtheorem{alemma}{Lemma}

\newenvironment{proof}{{\noindent \bf Proof:}}{\hfill$\Box$\medskip}
\definecolor{lred}{rgb}{1,0.8,0.8}
\definecolor{lblue}{rgb}{0.8,0.8,1}
\definecolor{dred}{rgb}{0.6,0,0}
\definecolor{dblue}{rgb}{0,0,0.5}
\definecolor{dgreen}{rgb}{0,0.5,0.5}

\title{Computation of graphical derivatives of normal cone maps to
 a class of conic constraint sets\footnote{Supported by the National Natural Science Foundation of China under project No.11571120 and the Natural Science Foundation of Guangdong Province under project No.2015A030313214.}}
  \author{Yulan Liu\footnote{School of Applied Mathematics, Guangdong University of Technology, Guangzhou.},
 \ Ying Sun\footnote{ School of Mathematics, South China University of Technology, Guangzhou.}
  \ \ {\rm and}\ \
 Shaohua Pan\footnote{Corresponding author(shhpan@scut.edu.cn), School of Mathematics, South China University of Technology, Guangzhou.}}

 \date{May 5, 2018\\ (revised version)}

\begin{document}

 \maketitle

 \begin{abstract}
  This paper concerns with the graphical derivative of the normals to
  the conic constraint $g(x)\in\!K$, where $g\!:\mathbb{X}\to\mathbb{Y}$
  is a twice continuously differentiable mapping and $K\subseteq\mathbb{Y}$ is a nonempty
  closed convex set assumed to be $C^2$-cone reducible. Such a generalized derivative
  plays a crucial role in characterizing isolated calmness of the solution maps to
  generalized equations whose multivalued parts are modeled via the normals to
  the nonconvex set $\Gamma=g^{-1}(K)$. The main contribution of this paper
  is to provide an exact characterization for the graphical derivative of
  the normals to this class of nonconvex conic constraints under an assumption
  without requiring the nondegeneracy of the reference point as the papers
  \cite{Gfrerer17,Mordu15,Mordu151} do.
 \end{abstract}



%----------------------------------------------------------------------------------------------
 \section{Introduction}\label{sec1}

 Generalized derivatives introduced in modern variational analysis represent
 an efficient tool to study stability analysis of multifunctions, especially
 the so-called solution maps associated with parameter-dependent variational
 inequalities or generalized equations; see Rockafellar and Wets \cite{RW98},
 Klatte and Kummer \cite{KK02}, Mordukhovich \cite{Mordu06}, and Dontchev
 and Rockafellar \cite{DR09}. The stability properties of the solution maps
 to generalized equations, whose multivalued parts are modelled via regular
 normals to the polyhedral conic constraints, have been analyzed in the seventies,
 above all in the papers by Robinson \cite{Robinson79,Robinson81,Robinson82},
 and an overview of available results in this setting can be found in
 Klatte and Kummer \cite{KK02} and Dontchev and Rockafellar \cite[Chapter 2E]{DR09}.
 In the recent decade, some active research is given to the stability properties
 of the solution maps to those generalized equations associated with nonpolyhedral conic
 constraints \cite{BS00}, such as positive semidefinite conic constraints \cite{Sun06,ZhangZ16},
 Lorentz conic constraints \cite{Outrata11,BR05,HMordu17}, and more general constraints
 associated with cone reducible closed convex sets \cite{DingSZ17,KK13,Mordu151}.

 \medskip

 Let $\mathbb{X}, \mathbb{Y}$ and $\mathbb{P}$ be finite dimensional
 vector spaces endowed with the inner product $\langle \cdot,\cdot\rangle$
 and its induced norm $\|\cdot\|$. Let $g\!:\mathbb{X}\to\mathbb{Y}$ be
 a twice continuously differentiable mapping, and let $K\subseteq\!\mathbb{Y}$
 be a nonempty closed convex set which is assumed to be $C^2$-cone reducible.
 The class of $C^2$-cone reducible sets is rich, including all the polyhedral
 convex sets and many non-polyhedral sets such as the second-order cone
 \cite[Lemma 15]{BR05}, the positive semidefinite cone \cite[Example 3.140]{BS00},
 and the epigraph cone of the Ky Fan matrix $k$-norm \cite{Ding12}. Moreover,
 the Cartesian product of $C^2$-cone reducible sets is also $C^2$-cone reducible
 \cite{Shapiro03}.
 This paper focuses on the computation of the graphical derivative of the normal
 cone mappings to the conic constraint $g(x)\!\in K$ or equivalently the set
 \begin{equation}\label{Gamma}
  \Gamma:=g^{-1}(K),
 \end{equation}
 which is also the set of the zeros to the following multifunction associated to $g(x)\in K$:
  \begin{equation}\label{MGmap}
  \mathcal{G}(x):=g(x)-K\quad{\rm for}\ x\in\mathbb{X}.
 \end{equation}
 Since our assumptions throughout this paper ensure that the regular
 and limiting normal cones to $\Gamma$ agree, we use the generic normal
 cone symbol $\mathcal{N}$ below; see Section \ref{sec2} for details.

 \medskip

 The present study, being certainly of its own interest, is motivated by
 the subsequent application to the characterization of the isolated calmness property
 for parameterized equilibria represented as
 the solution map to the following generalized equation (GE)
 \begin{equation}\label{GE}
   0\in F(p,x)+\mathcal{N}_\Gamma(x),
 \end{equation}
 where $F\!:\mathbb{P}\times\mathbb{X}\to\mathbb{X}$ is a locally Lipschitz
 and directionally differentiable mapping, and $\mathcal{N}_\Gamma$
 is the regular normal cone mapping to the set $\Gamma$.
 The solution map of \eqref{GE} is given by
 \begin{equation}\label{MSmap}
  \mathcal{S}(p):=\big\{x\in\mathbb{X}\ |\ 0\in F(p,x)+\mathcal{N}_\Gamma(x)\big\}.
 \end{equation}
 To achieve this goal, motivated by the crucial result due to King and Rockafellar \cite{KR92}
 or Levy \cite{Levy96}, we need to compute the graphical derivative of $\mathcal{S}$
 in terms of the initial problem data of \eqref{GE} and the corresponding values
 at the reference solution point. This amounts to developing the expression of
 the graphical derivative of the normal cone mapping $\mathcal{N}_\Gamma$.
 In addition, the expression of the graphical derivative of $\mathcal{N}_\Gamma$
 is also helpful to the characterization of the regular and limiting normals to
 $\mathcal{N}_\Gamma$.

 \medskip

 When the set $\Gamma$ is convex and the mapping $F$ is continuously differentiable,
 Mordukhovich et al. \cite{Mordu151} provided a formula for calculating
 the graphical derivative of $\mathcal{S}$.
 Recognizing that the convexity assumption on $\Gamma$ is very restrictive,
 they later derived a second-order formula in \cite{Mordu15} for calculating
 the graphical derivative of the regular normal $\widehat{\mathcal{N}}_\Gamma$
 and then that of the solution map $\mathcal{S}$ in terms of Lagrange multipliers
 of the perturbed KKT system and the critical cone of $K$, under the projection
 derivation condition (PDC) on $K$ at a nondegenerate reference point.
 Although the PDC relaxes the polyhedrality assumption imposed on the set $K$
 by \cite{Henrion13}, it actually requires that $K$ has similar properties as
 a polyhedral set does; for example, the PDC holds under the second-order extended
 polyhedricity condition from \cite{BS00}. When $K$ is non-polyhedral convex cone,
 although the PDC always holds at the vertex, the popular positive semidefinite cone
 and Lorentz cone generally do not satisfy this condition at nonzero vertexes
 (see \cite[Corollary 3.5]{HMordu17}). In addition, Gfrerer and Outrata \cite{Gfrerer17}
 also derived a formula for calculating the graphical derivative of the regular normal
 $\widehat{\mathcal{N}}_\Gamma$ by imposing the nondegeneracy of
 the reference point and a weakened version of the reducibility.
 The nondegeneracy of the reference point is strong, and the papers mentioned above
 all require this assumption.

 \medskip

 Recently, for the case where $K$ is the Lorentz cone,
 Hang, Mordukhovich and Sarabi \cite{HMordu17} fully exploited
 the structure of the Lorentz cone and precisely calculated the graphical
 derivative of the normal cone mapping to $\widehat{\mathcal{N}}_\Gamma$
 under an assumption even weaker than the one used in \cite{Gfrerer16-MOR}
 to compute the graphical derivative of $\widehat{\mathcal{N}}_\Gamma$
 with $K=\mathbb{R}_{-}^m$; and for optimization problems with the conic
 constraint $g(x)\in K$, Ding, Sun and Zhang \cite{DingSZ17} verified that
 the KKT solution mapping is robustly isolated calm iff both the strict Robinson constraint
 qualification (SRCQ) and the second order sufficient condition hold. Their results,
 to a certain extent, imply that it is possible to achieve the exact characterization
 for the graphical derivative of $\mathcal{N}_\Gamma$ without
 requiring the nondegeneracy.

 \medskip

 Recall that the SRCQ for the system $g(x)\!\in K$ is said to hold at
 $\overline{x}$ with respect to (w.r.t.) some multiplier
 $\overline{\lambda}\in\mathcal{N}_K(g(\overline{x}))$ if
 \begin{equation}\label{SRCQ}
   g'(\overline{x})\mathbb{X}+\mathcal{T}_{K}(g(\overline{x}))\cap[\![\overline{\lambda}]\!]^{\perp}
   =\mathbb{Y},
 \end{equation}
 which is weaker than the nondegeneracy of $\overline{x}$ w.r.t. the mapping $g$
 and the set $K$:
 \begin{equation}\label{Nondegeneracy}
   g'(\overline{x})\mathbb{X}+{\rm lin}\big[\mathcal{T}_{K}(g(\overline{x}))\big]
   =\mathbb{Y}.
 \end{equation}
 In this work we shall provide an exact characterization for the graphical
 derivative of $\mathcal{N}_{\Gamma}$ under the metric subregularity
 of $\mathcal{G}$ and a multifunction $\Phi$ (see \eqref{Phimap} for its definition)
 and the SRCQ for the system $g(x)\in K$. Among others, the metric
 subregularity of $\Phi$ is only used for deriving the lower estimation
 for the graphical derivative of $\mathcal{N}_{\Gamma}$, while the SRCQ
 for the system $g(x)\in K$ is used for achieving the upper estimation.
 Since our upper estimation only requires the SRCQ for the system $g(x)\in K$,
 one can achieve the isolated calmness of $\mathcal{S}$ without the nondegeneracy.
 During the reviewing of this paper, we learned that Gfrerer and Mordukhovich
 \cite{Gfrerer171} skillfully derived the lower estimation for the graphical
 derivative of $\mathcal{N}_{\Gamma}$ only under the metric subregularity of $\mathcal{G}$,
 which is a trivial assumption. Although their exact characterization for
 the graphical derivative of $\mathcal{N}_{\Gamma}$ does not require
 the uniqueness of the multipliers, one needs to solve
 a linear conic optimization problem to achieve the required multiplier.
 Moreover, their formula involves the normal cone of the critical cone
 of $\Gamma$, which has a workable expression only under the closedness of
 the radial cone to $\mathcal{N}_{\Gamma}$ (see Proposition \ref{critical-normal-prop1}
 and \ref{conditions-closedness}). In other words, under the uniqueness of
 the multipliers and the closedness of the radial cone to $\mathcal{N}_{\Gamma}$,
 their formula for the graphical derivative of $\mathcal{N}_{\Gamma}$ agrees
 with ours. As direct applications of this result, we establish a lower estimation
 for the regular coderivative of $\mathcal{N}_\Gamma$ under the SRCQ,
 and an upper estimation for the coderivative of $\widehat{\mathcal{N}}_\Gamma$
 under the metric subregularity of $\Phi$, which partly improves the results of
 \cite[Theorem 7]{Outrata11} and \cite[Theorem 4.1]{Mordu15}.

 \medskip

 Our notation is basically standard. A hollow capital, say $\mathbb{Z}$,
 denotes a finite dimensional vector space endowed with the inner product
 $\langle \cdot,\cdot\rangle$ and its induced norm $\|\cdot\|$, and
 $\mathbb{B}_{\mathbb{Z}}$ means the closed unit ball centered at
 the origin in $\mathbb{Z}$. For a given $z\in\mathbb{Z}$, $\mathbb{B}(z,\delta)$
 means the closed ball of radius $\delta$ centered at $z$ in $\mathbb{Z}$.
 For a given closed convex set $\Omega$, $\Pi_{\Omega}$ denotes the projection
 operator onto $\Omega$; and for a given nonempty convex cone $\mathcal{K}$,
 $\mathcal{K}^{\circ}$ means the negative polar of $\mathcal{K}$.
 For a linear operator $\mathcal{A}$, $\mathcal{A}^*$ denotes the adjoint of $\mathcal{A}$.
 For a given vector $z$, the notation $[\![z]\!]$ denotes the subspace generated by $z$.

%-----------------------------------------------------------------------------------------------
 \section{Preliminaries}\label{sec2}

 This section provides some background knowledge and some necessary results.
 Let $\Omega\subseteq\mathbb{Z}$ be a nonempty set. For a fixed $\overline{z}\in\Omega$,
 from \cite{BS00} the radial cone to $\Omega$ at $\overline{z}$ is defined by
 \[
   \mathcal{R}_{\Omega}(\overline{z}):=\big\{h\in \mathbb{Z}\ |\ \exists\,t^*>0\ {\rm such\ that\ for\ all}\
   t\in [0,t^*],\,\overline{z}+th\in \Omega\big\},
 \]
 while from \cite{RW98} the contingent cone to $\Omega$ at $\overline{z}$ is defined by
 \[
   \mathcal{T}_{\Omega}(\overline{z}):=\big\{w\in\mathbb{Z}\ |\ \exists t_k\downarrow 0,\,
   w^k\to w\ {\rm with}\ \overline{x}+t_kw_k\in\Omega\big\}.
 \]
  Notice that $\mathcal{R}_{\Omega}(\overline{z})\subseteq \mathcal{T}_{\Omega}(\overline{z})$,
  and when $\Omega$ is convex, $\mathcal{T}_{\Omega}(\overline{z})={\rm cl}(\mathcal{R}_{\Omega}(\overline{z}))$.
 For a fixed $\overline{z}\in\Omega$, by \cite{RW98} the regular normal cone
 to $\Omega$ at $\overline{z}$ is defined by
 \[
   \widehat{\mathcal{N}}_{\Omega}(\overline{z})
   :=\Big\{v\in\mathbb{Z}\ |\ \limsup_{z\xrightarrow[\Omega]{}\overline{z}}
   \frac{\langle v,z-\overline{z}\rangle}{\|z-\overline{z}\|}\le 0\Big\},
 \]
 and the basic/limiting normal cone to $\Omega$ at $\overline{z}$ admits the following representation
 \[
   \mathcal{N}_\Omega(\overline{z})=\limsup_{z\xrightarrow[\Omega]{}\overline{z}}\widehat{\mathcal{N}}_{\Omega}(z),
 \]
 which, if $\Omega$ is locally closed at $\overline{z}\in\Omega$, is equivalent to
 the original definition by Mordukhovich \cite{Mordu76}, i.e.,
 \(
   \mathcal{N}_\Omega(\overline{z})\!:=\limsup_{z\to\overline{z}}\big[{\rm cone}(z-\Pi_{\Omega}(z))\big].
 \)
 Notice that $\widehat{\mathcal{N}}_{\Omega}(\overline{z})=(\mathcal{T}_{\Omega}(\overline{z}))^\circ$,
 and when $\Omega$ is convex, $\widehat{\mathcal{N}}_{\Omega}(\overline{z})=\mathcal{N}_{\Omega}(\overline{z})$.
 Given a direction $h\in\mathbb{Z}$, the directional limiting normal cone to $\Omega$
 at $\overline{x}$ in $h$ is defined by
 \(
   \mathcal{N}_{\Omega}(\overline{x};h)
   :=\limsup_{t\searrow 0,h'\to h}\widehat{\mathcal{N}}_{\Omega}(\overline{x}+th').
 \)
 Various properties of the directional limiting normal cone can be found
 in \cite{Gfrerer13,Gfrerer16}.
%--------------------------------------------------------------------------------------------
 \subsection{Lipschitz-type properties of a multifunction}\label{subsec2.1}

 Let $\mathcal{F}\!:\mathbb{Z}\rightrightarrows\mathbb{W}$ be a given multifunction.
 Consider an arbitrary $(\overline{z},\overline{w})\in{\rm gph}\mathcal{F}$
 such that $\mathcal{F}$ is locally closed at $(\overline{z},\overline{w})$.
 We recall from \cite{RW98,DR09} several Lipschitz-type properties of
 the multifunction $\mathcal{F}$,
 including the Aubin property, the calmness and the isolated calmness.
%-----------------------------------------------------------------------------------------
 \begin{definition}\label{Aubin-def}
  The multifunction $\mathcal{F}$ is said to have the Aubin property at
  $\overline{z}$ for $\overline{w}$ if there exists $\kappa\ge 0$ along with
  $\varepsilon>0$ and $\delta>0$ such that
  for all $z,z'\in\mathbb{B}(\overline{z},\varepsilon)$,
  \[
    \mathcal{F}(z)\cap\mathbb{B}(\overline{w},\delta)\subseteq\mathcal{F}(z')
    +\kappa\|z-z'\|\mathbb{B}_{\mathbb{W}}.
  \]
 \end{definition}
 %-------------------------------------------------------------------------------------------------
 \begin{definition}\label{calm-def}
  The multifunction $\mathcal{F}$ is said to be calm at $\overline{z}$ for $\overline{w}$
  if there exists $\kappa\ge0$ along with $\varepsilon>0$ and $\delta>0$ such that for all
  $z\in\mathbb{B}(\overline{z},\varepsilon)$,
  \begin{equation}\label{calm-inclusion1}
    \mathcal{F}(z)\cap\mathbb{B}(\overline{w},\delta)
    \subseteq\mathcal{F}(\overline{z})+\kappa\|z-\overline{z}\|\mathbb{B}_{\mathbb{W}};
  \end{equation}
  if in addition $\mathcal{F}(\overline{z})\cap \mathbb{B}(\overline{w},\delta) =\{\overline{w}\}$,
  $\mathcal{F}$ is said to be isolated calm at $\overline{z}$ for $\overline{w}$.
 \end{definition}

  The coderivative and graphical derivative of $\mathcal{F}$ are the convenient tools
  to study the Aubin property and isolated calmness of $\mathcal{F}$, respectively.
  Recall from \cite{Mordu80,Aubin81} that the coderivative of $\mathcal{F}$ at $\overline{z}$
  for $\overline{w}\in\mathcal{F}(\overline{z})$ is the mapping
  $D^*\mathcal{F}(\overline{z}|\overline{w})\!:\mathbb{W}\rightrightarrows\mathbb{Z}$
  defined by
  \[
    \Delta z\in D^*\mathcal{F}(\overline{z}|\overline{w})(\Delta w)\Longleftrightarrow
    (\Delta z,-\Delta w)\in\mathcal{N}_{{\rm gph}\mathcal{F}}(\overline{z},\overline{w}),
  \]
  and the graphical derivative of $\mathcal{F}$ at $(\overline{z},\overline{w})$
  is the mapping $D\mathcal{F}(\overline{z}|\overline{w})\!:\mathbb{Z}\rightrightarrows\mathbb{W}$
  defined by
  \[
    \Delta w\in D\mathcal{F}(\overline{z}|\overline{w})(\Delta z)\Longleftrightarrow
    (\Delta z,\Delta w)\in\mathcal{T}_{{\rm gph}\mathcal{F}}(\overline{z},\overline{w}).
  \]
  With the coderivative and graphical derivative of $\mathcal{F}$,
  we have the following conclusions.
%--------------------------------------------------------------------------------------------
  \begin{lemma}\label{chara-Aubin}(see \cite[Theorem 5.7]{Mordu93}
  or \cite[Theorem 9.40]{RW98})\ The multifunction $\mathcal{F}$ has
  the Aubin property at $\overline{z}$ for $\overline{w}$
  if and only if $D^*\mathcal{F}(\overline{z}|\overline{w})(0)=\{0\}$.
  \end{lemma}
  \begin{lemma}\label{chara-icalm}(see \cite[Proposition 2.1]{KR92}
  or \cite[Proposition 4.1]{Levy96})\ The multifunction $\mathcal{F}$
  is isolated calm at $\overline{z}$ for $\overline{w}$
  if and only if $D\mathcal{F}(\overline{z}|\overline{w})(0)=\{0\}$.
  \end{lemma}

  Next we recall from \cite{RW98,DR09} metric regularity and metric subregularity,
  respectively.
%-----------------------------------------------------------------------------------------
 \begin{definition}\label{regular-def}
  The multifunction $\mathcal{F}$ is said to be metrically regular at $\overline{z}$
  for $\overline{w}$ if there exists $\kappa\ge 0$ along with $\varepsilon>0$ and $\delta>0$ such that
  for all $z\in\mathbb{B}(\overline{z},\varepsilon)$ and $w\in\mathbb{B}(\overline{w},\delta)$,
  \[
     {\rm dist}\big(z,\mathcal{F}^{-1}(w)\big)
     \le \kappa\,{\rm dist}\big(w,\mathcal{F}(z)\big).
  \]
 \end{definition}
 %------------------------------------------------------------------------------------------------------
 \begin{definition}\label{subregular-def}
  The multifunction $\mathcal{F}$ is said to be metrically subregular at $\overline{z}$
  for $\overline{w}$ if there exists $\kappa\ge0$ along with $\varepsilon>0$ and $\delta>0$
  such that for all $z\in\mathbb{B}(\overline{z},\varepsilon)$,
  \[
     {\rm dist}\big(z,\mathcal{F}^{-1}(\overline{w})\big)
     \le \kappa\,{\rm dist}\big(\overline{w},\mathcal{F}(z)\cap\mathbb{B}(\overline{w},\delta)\big).
  \]
 \end{definition}
 \begin{remark}\label{remark21}
  It is known that $\mathcal{F}$ has the Aubin property at $\overline{z}$ for $\overline{w}$
  iff $\mathcal{F}^{-1}$ is metrically regular at $\overline{w}$ for $\overline{z}$
  (see \cite{RW98,DR09}); and $\mathcal{F}$ is calm at $\overline{z}$ for $\overline{w}$
  iff $\mathcal{F}^{-1}$ is metrically subregular at $\overline{w}$ for $\overline{z}$
  (see \cite[Theorem 3H.3]{DR09}). By \cite[Exercise 3H.4]{DR09}, the restriction on
  $z\in\mathbb{B}(\overline{z},\varepsilon)$ in Definition \ref{calm-def} and
  the neighborhood $\mathbb{B}(\overline{w},\delta)$
  in Definition \ref{subregular-def} can be removed.
 \end{remark}

  The following lemma states a link between the graphical derivative
  of $\mathcal{F}$ and the contingent cone to the value of $\mathcal{F}$
  at some point, where the first part is easily proved by the definition,
  and the second part follows from \cite[Corollary 4.2]{Gfrerer16} and Remark \ref{remark21}.
%-----------------------------------------------------------------------------
 \begin{lemma}\label{TF-relation}
  For the multifunction $\mathcal{F}$ and the point $(\overline{z},\overline{w})$,
  $\mathcal{T}_{\mathcal{F}(\overline{z})}(\overline{w})\subseteq
  D\mathcal{F}(\overline{z}|\overline{w})(0)$. The converse inclusion also holds
  provided that $\mathcal{F}$ is calm at $\overline{z}$ for $\overline{w}$.
 \end{lemma}
 \subsection{Normal cone mapping to $C^{2}$-cone reducible set}\label{subsec2.2}

  We shall establish the calmness of the normal cone map to
  a $C^{2}$-cone reducible closed convex set, which is a nonpolyhedral
  counterpart of the seminal upper-Lipschitzian result by Robinson \cite{Robinson81}
  for convex polyhedral sets. First, we recall the $C^{\ell}$-cone reducibility.
%----------------------------------------------------------------------------------------------------
 \begin{definition}\label{cone-reduce}(\cite[Definition 3.135]{BS00})
  A closed convex set $\Omega$ in $\mathbb{Y}$ is said to be $C^{\ell}$-cone
  reducible at $\overline{y}\in\Omega$, if there exist an open neighborhood
  $\mathcal{Y}$ of $\overline{y}$, a pointed closed convex cone
  $\mathcal{D}\subseteq\mathbb{Z}$ and an $\ell$-times continuously
  differentiable mapping $\Xi\!: \mathcal{Y}\to\mathbb{Z}$ such that (i) $\Xi(\overline{y})=0$;
  (ii) $\Xi'(\overline{y})\!:\mathbb{Y}\to\mathbb{Z}$ is onto;
  (iii) $\Omega\cap\mathcal{Y}=\{y\in\mathcal{Y}\ |\  \Xi(y)\in\mathcal{D}\}$.
  We say that the closed convex set $\Omega$ is $C^\ell$-cone reducible
  if $\Omega$ is $C^\ell$-cone reducible at every $y\in\Omega$.
 \end{definition}
%-------------------------- Subregularity and calmness for N_K --------------------------
 \begin{theorem}\label{NK-calm}
  Let $\Omega\subseteq\mathbb{Y}$ be a closed convex set. Suppose that $\Omega$ is
  $C^{2}$-cone reducible at $\overline{y}\in \Omega$. Then, the normal cone mapping
  $\mathcal{N}_\Omega$ is calm at $\overline{y}$ for each $\overline{z}\in\mathcal{N}_\Omega(\overline{y})$.
 \end{theorem}
 \begin{proof}
  Since $\Omega$ is $C^{2}$-cone reducible at $\overline{y}\in \Omega$,
  there exist an open neighborhood $\mathcal{Y}$ of $\overline{y}$,
  a pointed closed convex cone $\mathcal{D}\subseteq\mathbb{Z}$,
  and a twice continuously differentiable $\Xi\!:\mathcal{Y}\to\mathbb{Z}$
  satisfying (i)-(iii) in Definition \ref{cone-reduce}.
  Since $\Xi'(\overline{y})\!:\mathbb{Y}\to\mathbb{Z}$ is onto,
  there exists $\varepsilon>0$ such that for each
  $y\in\mathbb{B}(\overline{y},\varepsilon)\subset \mathcal{Y}$,
  the mapping $\Xi'(y)\!:\mathbb{Y}\to\mathbb{Z}$ is onto.
  By \cite[Exercise 6.7]{RW98},
  \begin{equation}\label{ncone-Omega}
   \mathcal{N}_{\Omega}(y)=\nabla\Xi(y)\mathcal{N}_{\mathcal{D}}(\Xi(y))
   \quad\ \forall y\in\mathbb{B}(\overline{y},\varepsilon).
  \end{equation}
  Define $\mathcal{E}(y):=(\Xi'(y)\nabla\Xi(y))^{-1}\Xi'(y)$ for $y\in\mathcal{Y}$.
  Notice that the functions $\mathcal{E}(\cdot)$ and $\nabla\Xi(\cdot)$
  are continuously differentiable in $\mathcal{Y}$. There exist $\varepsilon'>0$,
  $L_{\mathcal{E}}>0$ and $L>0$ such that
  \begin{equation}\label{Lip-bound}
    \|\mathcal{E}(y)-\!\mathcal{E}(y')\|\le L_{\mathcal{E}}\|y-y'\|
    \ \ {\rm and}\ \
    \|\nabla\Xi(y)\!-\!\nabla\Xi(y')\|\le L\|y-y'\|\quad\forall y\in\mathbb{B}(\overline{y},\varepsilon').
  \end{equation}
  Now fix an arbitrary $\overline{z}\in\mathcal{N}_\Omega(\overline{y})$.
  In order to establish the calmness of $\mathcal{N}_\Omega$ at $\overline{y}$
  for $\overline{z}$, it suffices to argue that there exist $\overline{\varepsilon}>0$,
  $\overline{\delta}>0$ and $\overline{\kappa}>0$ such that
  for all $y\in\mathbb{B}(\overline{y},\overline{\varepsilon})$,
  \begin{equation}\label{aim-ineq1}
   \mathcal{N}_\Omega(y)\cap\mathbb{B}(\overline{z},\overline{\delta})
   \subseteq \mathcal{N}_\Omega(\overline{y})+\overline{\kappa}\|y-\overline{y}\|\mathbb{B}_{\mathbb{Y}}.
  \end{equation}
  Fix an arbitrary $\overline{\delta}\in(0,1)$ and set
  $\overline{\varepsilon}:=\frac{1}{2}\min\big\{\varepsilon,\varepsilon'\big\}$.
  Fix an arbitrary point $y\in\mathbb{B}(\overline{y},\overline{\varepsilon})$.
  If $\mathcal{N}_\Omega(y)\cap\mathbb{B}(\overline{z},\overline{\delta})=\emptyset$,
  the inclusion \eqref{aim-ineq1} automatically holds. So, we only need to consider
  the case where $\mathcal{N}_\Omega(y)\cap\mathbb{B}(\overline{z},\overline{\delta})\ne\emptyset$.
  Take an arbitrary $z\in \mathcal{N}_\Omega(y)\cap \mathbb{B}(\overline{z},\overline{\delta})$.
  From \eqref{ncone-Omega}, there exists $\xi\in \mathcal{N}_{\mathcal{D}}(\Xi(y))$
  such that $z=\nabla\Xi(y)\xi$. Since $\overline{z}\in\mathcal{N}_\Omega(\overline{y})$,
  there also exists $\overline{\xi}\in\mathcal{N}_{\mathcal{D}}(\Xi(\overline{y}))$
  such that $\overline{z}=\nabla\Xi(\overline{y})\overline{\xi}$.
  Clearly, $\xi=\mathcal{E}(y)z$ and $\overline{\xi}=\mathcal{E}(\overline{y})\overline{z}$. Then,
 \begin{align*}
   \|\xi-\overline{\xi}\|
   &=\|\mathcal{E}(y)z-\mathcal{E}(\overline{y})\overline{z}\|
   \le \|\mathcal{E}(y)z-\mathcal{E}(\overline{y})z\|+\|\mathcal{E}(\overline{y})z-\mathcal{E}(\overline{y})\overline{z}\|\nonumber\\
  &\le L_{\mathcal{E}}\|z\|\|y-\overline{y}\|+\|\mathcal{E}(\overline{y})\|\|z-\overline{z}\|
   \le L_{\mathcal{E}}(\|\overline{z}\|+\overline{\varepsilon})+\|\mathcal{E}(\overline{y})\|\overline{\delta}
   :=\widetilde{\delta}
 \end{align*}
  Since $\mathcal{D}\subseteq\mathbb{Z}$ is a pointed closed convex cone,
  we have $\mathcal{N}_{\mathcal{D}}(\Xi(y))\subseteq\mathcal{D}^{\circ}$
  and then $\xi\in\mathcal{D}^{\circ}$, which implies that $\nabla\Xi(\overline{y})\xi\in\nabla\Xi(\overline{y})\mathcal{N}_{\mathcal{D}}(\Xi(\overline{y}))
  =\mathcal{N}_\Omega(\overline{y})$. Thus,
  \begin{align*}
  {\rm dist}(z,\mathcal{N}_\Omega(\overline{y}))
  &={\rm dist}(\nabla\Xi(y)\xi,\mathcal{N}_\Omega(\overline{y}))
  \le\|\nabla\Xi(y)\xi-\nabla\Xi(\overline{y})\xi\|\\
  &\le\|\xi\|L\|y-\overline{y}\|
  \le L(\widetilde{\delta}+\|\overline{\xi}\|)\|y-\overline{y}\|.
  \end{align*}
  This shows that the inclusion \eqref{aim-ineq1} holds with
 $\overline{\kappa}=L(\widetilde{\delta}+\|\overline{\xi}\|)$.
 \end{proof}
%--------------------------------------------------------------------------------
 \begin{remark}\label{remark-NK}
  {\bf(a)} If $\Omega$ is a $C^{2}$-cone reducible closed convex cone with
  $\Omega^{\circ}=-\Omega$;
  for example, the positive semidefinite cone and Lorentz cone, then $\Omega^{\circ}$
  is a $C^{2}$-cone reducible closed convex cone. By Theorem \ref{NK-calm},
  the mapping $\mathcal{N}_{\Omega^{\circ}}$ is calm at each point of its graph. Along with
  $\mathcal{N}_{\Omega^{\circ}}=\mathcal{N}_{\Omega}^{-1}$, $\mathcal{N}_{\Omega}$
  is metrically subregular at each point of its graph. Thus, for
  a $C^{2}$-cone reducible closed convex cone $\Omega$ with $\Omega^{\circ}=-\Omega$,
  $\mathcal{N}_{\Omega}$ is both metrically subregular and calm at each point of its graph.
  This recovers the result of \cite[Proposition 3.3]{CuiST16}.

  \medskip
  \noindent
  {\bf(b)} When $\Omega$ is a closed nonconvex set, if there exists a closed cone
  $\mathcal{D}\subseteq\mathbb{Z}$ together with a twice continuously differentiable
  mapping $\Xi\!: \mathcal{Y}\to\mathbb{Z}$ such that (i) $\Xi(\overline{y})=0$;
  (ii) $\Xi'(\overline{y})\!:\mathbb{Y}\to\mathbb{Z}$ is onto;
  (iii) $\Omega\cap\mathcal{Y}=\{y\in\mathcal{Y}\ |\  \Xi(y)\in\mathcal{D}\}$,
  then from the proof of Theorem \ref{NK-calm} it follows that
  the regular normal cone mapping $\widehat{\mathcal{N}}_\Omega$ is calm at
  $\overline{y}$ for each $\overline{z}\in\widehat{\mathcal{N}}_\Omega(\overline{y})$..
 \end{remark}

  Now let $K$ be a closed convex set which is assumed to be $C^{2}$-cone reducible.
  By Theorem \ref{NK-calm}, its normal cone mapping $\mathcal{N}_K$ is calm at
  each $y\in K$ for $\lambda\in\mathcal{N}_K(y)$. From \cite[Proposition 3.136]{BS00},
  the set $K$ is also second-order regular at each $y\in K$,
  and hence $\mathcal{T}_{K}^{i,2}(y,h)=\mathcal{T}_{K}^{2}(y,h)$
  for any $h\in\mathbb{Y}$, where $\mathcal{T}_{K}^{i,2}(y,h)$
  and $\mathcal{T}_{K}^{2}(y,h)$ denote the inner and outer second order tangent sets
  to $K$ at $y$ in the direction $h$, respectively, defined by
  \begin{align*}
   \mathcal{T}_{K}^{i,2}(y,h):=\Big\{w\in \mathbb{Y}\;|\; {\rm dist}(y+th+\frac{1}{2}t^2w,K)=o(t^2), t\geq 0\Big\},\qquad\\
   \mathcal{T}_{K}^{2}(y,h):=\Big\{w\in \mathbb{Y}\;|\; \exists\;t_n\downarrow 0 {\;\rm such\;that\;} {\rm dist}(y+t_nh+\frac{1}{2}t_n^2 w,K)=o(t_n^2)\Big\}.
  \end{align*}
  From the standard reduction approach in \cite[Section 3.4.4]{BS00},
  we have the following result on the representation of the normal cone
  $\mathcal{N}_K$ and the ``sigma term'' of $K$.
%----------------------------------------------------------------------------------------------------
  \begin{lemma}\label{lemma-reduction}
   Let $\overline{y}\in K$ be given. Then there exist an open neighborhood $\mathcal{Y}$
   of $\overline{y}$, a pointed closed convex cone $D\subseteq\mathbb{Z}$
   and a twice continuously differentiable mapping $\Xi\!:\mathcal{Y}\to\mathbb{Z}$
   satisfying conditions (i)-(iii) in Definition \ref{cone-reduce} such that
   for any $y\in\mathcal{Y}$,
   \begin{equation}\label{normal-cone}
     \mathcal{N}_{K}(y)=\nabla\Xi(y)\mathcal{N}_{D}(\Xi(y));
   \end{equation}
   and for any $\lambda\in\mathcal{N}_{K}(y)$ there exists a unique
   $u\in\!\mathcal{N}_{D}(\Xi(y))$ such that $\lambda=\nabla\Xi(y)u$ and
   \begin{equation}\label{Upsilon}
     \Upsilon(h):=-\sigma\big(\lambda,\mathcal{T}_{K}^2(y,h)\big)
     =\langle u,\Xi''(y)(h,h)\rangle
     \quad \forall h\in\mathcal{C}_{K}(y,\lambda)
   \end{equation}
  where $\sigma(\cdot,\mathcal{T}_{K}^2(y,h))$ is the support function
  of $\mathcal{T}_{K}^2(y,h)$, and for any $y\in K$,
  $\mathcal{C}_{K}(y,\lambda)$ is the critical cone of $K$ at $y$
  with respect to $\lambda\in\!\mathcal{N}_{K}(y)$, defined as
  \(
    \mathcal{C}_{K}(y,\lambda):=\mathcal{T}_{K}(y)\cap [\![\lambda]\!]^{\perp}.
  \)
  \end{lemma}

  Next we recall a useful result on the directional derivative of the projection
  operator $\Pi_{K}$. Fix an arbitrary $y\in\mathbb{Y}$. Write $\overline{y}:=\Pi_K(y)$
  and take $\overline{\lambda}\in\mathcal{N}_{K}(\overline{y})$.
  Since $K$ is second-order regular at $\overline{y}$, by \cite[Theorem 7.2]{BCS98}
  the mapping $\Pi_{K}$ is directionally differentiable at $y$ and the directional derivative
  $\Pi_{K}'(y;h)$ for any direction $h\in\mathbb{Y}$ satisfies
  \[
    \Pi_{K}'(y;h)=\mathop{\arg\min}_{d\in\mathcal{C}_{K}(\overline{y},\overline{\lambda})}
    \Big\{\|d-h\|^2-\sigma\big(\overline{\lambda},\mathcal{T}_{K}^2(\overline{y},d)\big)\Big\}.
  \]
  In addition, by following the arguments as those for \cite[Theorem 3.1]{WZZhang14},
  one can obtain
  \[
    \mathcal{T}_{{\rm gph}\mathcal{N}_K}(\overline{y},\overline{\lambda})=
    \big\{(\Delta z,\Delta w)\in\mathbb{Y}\times\mathbb{Y}\ |\
     \Pi_K'(\overline{y}+\overline{\lambda};\Delta z+\Delta w)=\Delta z\big\}.
  \]
  Combining this with \cite[Lemma 10]{DingSZ17}, we have the following
  conclusion for the graphical derivative of $\mathcal{N}_K$,
  the directional derivative of $\Pi_{K}$ and the critical cone of the set $K$.
%------------------------------------------------------------------------------------------------
  \begin{lemma}\label{dir-proj}
   Consider an arbitrary point pair $(\overline{y},\overline{\lambda})\in{\rm gph}\mathcal{N}_K$,
   and write $y:=\overline{y}+\overline{\lambda}$. Then,
   with $\Upsilon(\cdot)=-\sigma\big(\overline{\lambda},\mathcal{T}_{K}^2(\overline{y},\cdot)\big)
   =\langle u,\Xi''(\overline{y})(\cdot,\cdot)\rangle$ for $u\in\mathcal{N}_D(\Xi(\overline{y}))$,
   it holds that
   \begin{align}\label{system-regular}
    \Delta\lambda\in D\mathcal{N}_K(\overline{y}|\overline{\lambda})(\Delta y)
    &\Longleftrightarrow \Delta y-\Pi_{K}'(y;\Delta y\!+\!\Delta\lambda)=0\nonumber\\
    &\Longleftrightarrow
    \left\{\begin{array}{ll}
      \Delta y\in\mathcal{C}_{K}(\overline{y},\overline{\lambda}),\\
      \Delta\lambda-\frac{1}{2}\nabla\Upsilon(\Delta y)\in\big[\mathcal{C}_{K}(\overline{y},\overline{\lambda})\big]^{\circ},\\
      \langle\Delta y,\Delta\lambda\rangle
     =-\sigma(\overline{\lambda},\mathcal{T}_{K}^2(\overline{y},\Delta y)).
     \end{array}\right.
   \end{align}
 \end{lemma}
%------------------------------------------------------------------------------------
 \subsection{Contingent and normal cones to a composite set}\label{subsec2.3}

  Consider a set $\Theta\!:=H^{-1}(\Delta)$ where $H\!:\mathbb{Z}\to\mathbb{W}$
  is a mapping and $\Delta\subseteq\mathbb{W}$ is a closed set.
  The following characterization holds for the contingent cone to the set $\Theta$.
%--------------------------------------------------------------------------------------------------
 \begin{lemma}\label{Tcone-lemma}
  Suppose $H$ is Lipschitz near $\overline{z}$ and directionally differentiable
  at $\overline{z}$. Then,
  \begin{equation}\label{Tcone-formula}
  \mathcal{T}_{\Theta}\big(\overline{z}\big)
  \subseteq\big\{h\in\mathbb{Z}\ |\ H'(\overline{z};h)\in\mathcal{T}_{\Delta}(H(\overline{z}))\big\}.
  \end{equation}
  If $\mathcal{H}(z):=H(z)-\Delta$ is metrically subregular at $\overline{z}$ for $0$,
  the converse inclusion also holds.
 \end{lemma}

 The first part of Lemma \ref{Tcone-lemma} follows by the definition
 of the contingent cone and the Hadamard directional differentiability
 of $H$, and the second part is by \cite[Proposition 1]{Henrion05}.
 Combining Lemma \ref{Tcone-lemma} with \cite[Page 211-212]{Ioffe08},
 we can obtain the following result.
 %-------------------------------------------------------------------------------------------
 \begin{lemma}\label{normal-cone-lemma}
  Consider an arbitrary $\overline{z}\in\Theta$. Let $\mathcal{H}$ be
  the multifunction defined in Lemma \ref{Tcone-lemma}. If $\mathcal{H}$
  is metrically subregular at $\overline{z}$ for $0$, then
  $\mathcal{N}_{\Theta}(\overline{z})\subseteq D^*H(\overline{z})\big[\mathcal{N}_\Delta(H(\overline{z}))\big]$.
  If in addition $H$ is strictly differentiable and
  $\widehat{\mathcal{N}}_\Delta(H(\overline{z}))=\mathcal{N}_\Delta(H(\overline{z}))$,
  then it holds that
  \[
    \widehat{\mathcal{N}}_{\Theta}(\overline{z})=\mathcal{N}_{\Theta}(\overline{z})
     =\big\{\nabla H(\overline{z})y\ |\ y\in\widehat{\mathcal{N}}_\Delta(H(\overline{z}))\big\}.
  \]
 \end{lemma}

  Recall that $\Gamma=g^{-1}(K)$ where the mapping $g$ and the closed convex set
  $K$ satisfy the standard assumption. By Lemma \ref{Tcone-lemma}-\ref{normal-cone-lemma},
  under the metric subregularity of $\mathcal{G}$, we have the following
  characterization for the contingent cone and normal cone to
  the set $\Gamma$.
%--------------------------------------------------------------------------------------------
 \begin{corollary}\label{TNcone-Gamma}
  Consider an arbitrary $\overline{x}\in\Gamma$. If the multifunction
  $\mathcal{G}$ defined by \eqref{MGmap} is metrically subregular at
  $\overline{x}$ for the origin,
  then it holds that
  \begin{align}{}\label{TGamma1}
   \mathcal{T}_{\Gamma}(\overline{x})
   =\!\big\{h\in\mathbb{X}\ |\ g'(\overline{x})h\in\mathcal{T}_{K}(g(\overline{x}))\big\},\quad\\
  \mathcal{N}_{\Gamma}(\overline{x})=\mathcal{\widehat{N}}_{\Gamma}(\overline{x})
   =\!\big\{\nabla g(\overline{x})\lambda\ |\ \lambda\in\mathcal{N}_{K}(g(\overline{x}))\big\}.
   \label{TGamma1}
  \end{align}
 \end{corollary}
%------------------------------------------------------------------------------------------
 \begin{remark}\label{robust-remark}
  By Definition \ref{subregular-def}, the metric subregularity of $\mathcal{G}$ at
  $\overline{x}\in\Gamma$ for $0$ is equivalent to requiring the existence of $\kappa\ge0$
  along with $\varepsilon>0$ such that for all $x\in\mathbb{B}(\overline{x},\varepsilon)$,
  \[
    {\rm dist}(x,\Gamma)\le \kappa{\rm dist}(g(x),K).
  \]
  As remarked in \cite{HMordu17}, this means that the metric subregularity of $\mathcal{G}$
  at $\overline{x}\in\Gamma$ is robust in the sense that if $\mathcal{G}$ is metrically subregular
  at $\overline{x}\in\Gamma$, then so is $\mathcal{G}$ at any $x\in\Gamma$ near $\overline{x}$.
 \end{remark}
%-------------------------------------------------------------------------------------------
 \subsection{Multiplier set map and critical cone to $\Gamma$}\label{subsec2.4}

  Consider $\Gamma=g^{-1}(K)$ again. By Corollary \ref{TNcone-Gamma},
  under the metric subregularity of $\mathcal{G}$, $\mathcal{N}_{\Gamma}$
  takes the form of \eqref{TGamma1}. In view of this, for any given
  $x\in\Gamma$ and $v\in\mathcal{N}_{\Gamma}(x)$, we define
  \[
    \mathcal{M}(x,v):=\big\{\lambda\in\mathcal{N}_K(g(x))\ |\ v=\nabla g(x)\lambda\big\}
  \]
  which is the multiplier set associated to $(x,v)$, and denote by
  $\mathcal{M}_x\!:\mathbb{X}\rightrightarrows\mathbb{Y}$ the localized
  version of the multiplier set mapping $\mathcal{M}$, that is,
  $\mathcal{M}_x$ has the following form
  \begin{equation}\label{MMapx}
   \mathcal{M}_{x}(v):=\big\{\lambda\in\mathcal{N}_K(g(x))\ |\ v=\nabla g(x)\lambda\big\}.
  \end{equation}
  Clearly, $\mathcal{M}_{x}$ is a closed convex multifunction.
  For $\mathcal{M}_{x}$, we have the following result.
%--------------------------------------------------------------------------------------------
 \begin{proposition}\label{prop-Mx}
  Consider an arbitrary point $x\in\Gamma$. For any given $(v,\lambda)\in{\rm gph}\mathcal{M}_{x}$,
  \begin{equation}\label{Tcone-Mx}
   \mathcal{T}_{{\rm gph}\mathcal{M}_{x}}(v,\lambda)
   =\big\{(\xi,\eta)\in\mathbb{X}\times\mathbb{Y}\ |\ \xi=\nabla g(x)\eta,\,\eta\in\mathcal{T}_{\mathcal{N}_{K}(g(x))}(\lambda)\big\},
  \end{equation}
  and hence $\mathcal{M}_{x}$ is isolated calm at $v$ for $\lambda$ iff
  one of the following conditions holds:
  \begin{align*}
   \!{\rm Ker}(\nabla g(x))\cap \mathcal{T}_{\mathcal{N}_{K}(g(x))}(\lambda)=\{0\}
   &\Longleftrightarrow{\rm Ker}(\nabla g(x))\cap D\mathcal{N}_{K}(g(x)|\lambda)(0)=\{0\},\\
   &\Longleftrightarrow{\rm SRCQ\ for\ the\ system}\ g(x)\in K\ {\rm at}\ x\ {\rm w.r.t.}\ \lambda.\nonumber
   \end{align*}
  \end{proposition}
 \begin{proof}
  Notice that ${\rm gph}\mathcal{M}_{x}=\mathcal{L}\big(\mathcal{N}_{K}(g(x))\big)$
  where $\mathcal{L}(u):=\left(\begin{matrix}
          \nabla\!g(x)u\\ u
          \end{matrix}\right)$ for $u\in\mathbb{Y}$. From the convexity of
  $\mathcal{N}_{K}(g(x))$ and the last part of \cite[Theorem 6.43]{RW98}, it follows that
  \[
    \mathcal{T}_{{\rm gph}\mathcal{M}_{x}}(v,\lambda)
    ={\rm cl}\big\{(\xi,\eta)\in\mathbb{X}\times\mathbb{Y}\ |\ \xi=\nabla g(x)\eta,\,\eta\in\mathcal{T}_{\mathcal{N}_{K}(g(x))}(\lambda)\big\}.
  \]
  Since
  \(
    \big\{(\xi,\eta)\in\mathbb{X}\times\mathbb{Y}\ |\ \xi=\nabla g(x)\eta,\eta\in\mathcal{T}_{\mathcal{N}_{K}(g(x))}(\lambda)\big\}
  \)
  is closed, the result in \eqref{Tcone-Mx} holds.
  By Lemma \ref{chara-icalm}, $\mathcal{M}_{x}$ is isolated calm at $v$ for $\lambda$
  iff $(0,\eta)\in\mathcal{T}_{{\rm gph}\mathcal{M}_{x}}(v,\lambda)$ implies $\eta=0$.
  Together with \eqref{Tcone-Mx}, this is equivalent to requiring
  that ${\rm Ker}(\nabla g(x))\cap\mathcal{T}_{\mathcal{N}_{K}(g(x))}(\lambda)=\{0\}$.
  By Theorem \ref{NK-calm} and Lemma \ref{TF-relation}, the first equivalence holds.
  Recall that the SRCQ for the system $g(x)\in K$ at $x$ w.r.t. $\lambda$ is requiring that
  \(
    g'(x)\mathbb{X}+\mathcal{T}_K(g(x))\cap[\![\lambda]\!]^{\perp}=\mathbb{X},
  \)
  which by \cite[Equations(2.31)\&(2.32)]{BS00} and \cite[Example 2.62]{BS00}
  is equivalent to saying that
  \[
    0={\rm Ker}(\nabla g(x))\cap{\rm cl}(\mathcal{N}_K(g(x))
    \!+[\![\lambda]\!])={\rm Ker}(\nabla g(x))
    \cap\mathcal{T}_{\mathcal{N}_{K}(g(x))}(\lambda).
  \]
  Thus, we obtain the second equivalence. The proof is then completed.
 \end{proof}

  Given $x\in\Gamma$ and $v\in\mathcal{N}_{\Gamma}(x)$, the critical cone to $\Gamma$
  at $x$ with respect to $v$ is defined as
  \[
    \mathcal{C}_{\Gamma}(x,v):=\mathcal{T}_{\Gamma}(x)\cap[\![v]\!]^{\perp}.
  \]
  By Corollary \ref{TNcone-Gamma}, under the metric subregularity of
  $\mathcal{G}$ at $x\in\Gamma$ for $0$, it holds that
  \begin{equation}\label{critial-equa1}
    \mathcal{C}_{\Gamma}(x,v)=[g'(x)]^{-1}\mathcal{C}_K(g(x),\lambda)
    \quad{\rm for\ each}\ \lambda\in\mathcal{M}_{x}(v).
  \end{equation}
  Next we provide a characterization for the normal cone to the critical cone of $\Gamma$.
%-----------------------------------------------------------------------------------------
 \begin{proposition}\label{critical-normal-prop1}
  Let $(\overline{x},\overline{v})\in{\rm gph}\mathcal{N}_{\Gamma}$.
  If $\mathcal{G}$ is metrically subregular at $\overline{x}$ for $0$, then
  \begin{align}\label{critical-cone1}
   \mathcal{N}_{\mathcal{C}_{\Gamma}(\overline{x},\overline{v})}(d)
   &\supseteq{\textstyle\bigcup_{\lambda\in\mathcal{M}_{\overline{x}}(\overline{v})}}
    \Big\{\nabla g(\overline{x})\xi\ |\ \langle\xi,g'(\overline{x})d\rangle=0,\,
                 \xi\in\mathcal{T}_{\mathcal{N}_{K}(g(\overline{x}))}(\lambda)\Big\}\\
   &\supseteq{\textstyle\bigcup_{\lambda\in\mathcal{M}_{\overline{x}}(\overline{v})}}
    \Big\{\nabla g(\overline{x})\xi\ |\ \langle\xi,g'(\overline{x})d\rangle=0,\,
                 \xi\in\mathcal{R}_{\mathcal{N}_{K}(g(\overline{x}))}(\lambda)\Big\}.
   \label{critical-cone2}
  \end{align}
  If, in addition, the radial cone $\mathcal{R}_{\mathcal{N}_{\Gamma}(\overline{x})}(\overline{v})$
  is closed, the inclusions become equality.
  \end{proposition}
  \begin{proof}
  Since $(\overline{x},\overline{v})\in{\rm gph}\mathcal{N}_{\Gamma}$ and
  $\mathcal{G}$ is metrically subregular at $\overline{x}$ for $0$,
  by Corollary \ref{TNcone-Gamma}, $\mathcal{M}_{\overline{x}}(\overline{v})\ne\emptyset$
  and $\mathcal{T}_{\Gamma}(\overline{x})$ is convex. The latter implies
  the convexity of $\mathcal{C}_{\Gamma}(\overline{x},\overline{v})$.
  Hence,
  \begin{equation}\label{temp-equa1-sec2}
   \mathcal{N}_{\mathcal{C}_{\Gamma}(\overline{x},\overline{v})}(d)
   =[\mathcal{C}_{\Gamma}(\overline{x},\overline{v})]^{\circ}\cap[\![d]\!]^{\perp}.
  \end{equation}
  The inclusion in \eqref{critical-cone2} is trivial, and
  we only need to establish the inclusion in \eqref{critical-cone1}.
  Let $h$ be an arbitrary point from the set on the right hand side of
  \eqref{critical-cone1}. Then there exist $\lambda\in\mathcal{M}_{\overline{x}}(\overline{v})$
  and $\xi\in\mathcal{T}_{\mathcal{N}_{K}(g(\overline{x}))}(\lambda)$ with
  $\langle\xi,g'(\overline{x})d\rangle=0$ such that $h=\nabla\!g(\overline{x})\xi$.
  From $\xi\in\mathcal{T}_{\mathcal{N}_{K}(g(\overline{x}))}(\lambda)$,
  there exist sequences $t_k\downarrow 0$ and $\xi^k\to\xi$ such that
  $\lambda+t_k\xi^k\in\mathcal{N}_K(g(\overline{x}))$ for each $k$.
  Fix an arbitrary $k\in\mathbb{N}$. For each $w\in\mathcal{C}_{\Gamma}(\overline{x},\overline{v})$,
  by \eqref{temp-equa1-sec2} it holds that
  \[
   0\ge\langle g'(\overline{x})w,\lambda+t_k\xi^k\rangle
   =\langle \overline{v},w\rangle+t_k\langle w,\nabla g(\overline{x})\xi^k\rangle
  =t_k\langle w,\nabla g(\overline{x})\xi^k\rangle,
  \]
  which implies that $\nabla g(\overline{x})\xi^k\in[\mathcal{C}_{\Gamma}(\overline{x},\overline{v})]^{\circ}$.
  Thus, $\nabla g(\overline{x})\xi\in[\mathcal{C}_{\Gamma}(\overline{x},\overline{v})]^{\circ}$.
  Together with $\langle\xi,g'(\overline{x})d\rangle=0$, we have
  $h=\nabla g(\overline{x})\xi\in[\mathcal{C}_{\Gamma}(\overline{x},\overline{v})]^{\circ}\cap[\![d]\!]^{\perp}$,
  and then $h\in\mathcal{N}_{\mathcal{C}_{\Gamma}(\overline{x},\overline{v})}(d)$ by
  \eqref{temp-equa1-sec2}. This shows that the set on the right hand side of \eqref{critical-cone1}
  is included in $\mathcal{N}_{\mathcal{C}_{\Gamma}(\overline{x},\overline{v})}(d)$.

  \medskip

  \medskip

  Assume that $\mathcal{R}_{\mathcal{N}_{\Gamma}(\overline{x})}(\overline{v})$
  is closed. To argue that the inclusions \eqref{critical-cone1} and \eqref{critical-cone2}
  become equality now, we only need to show that $\mathcal{N}_{\mathcal{C}_{\Gamma}(\overline{x},\overline{v})}(d)$
  is included in the set on the right hand hand side of \eqref{critical-cone2}.
  To this end, let $\overline{h}$ be an arbitrary point from $\mathcal{N}_{\mathcal{C}_{\Gamma}(\overline{x},\overline{v})}(d)$.
  Then
  \[
    [\mathcal{C}_{\Gamma}(\overline{x},\overline{v})]^{\circ}
    ={\rm cl}(\mathcal{N}_{\Gamma}(\overline{x})+[\![\overline{v}]\!])
    ={\rm cl}\big(\mathcal{R}_{\mathcal{N}_{\Gamma}(\overline{x})}(\overline{v})\big)
    =\mathcal{R}_{\mathcal{N}_{\Gamma}(\overline{x})}(\overline{v})
    =\mathcal{N}_{\Gamma}(\overline{x})+[\![\overline{v}]\!].
  \]
  where the first equality is by \cite[Equation (2.32)]{BS00},
  and the second is due to \cite[Example 2.62]{BS00}.
  Together with \eqref{temp-equa1-sec2},
  $\mathcal{N}_{\mathcal{C}_{\Gamma}(\overline{x},\overline{v})}(d)
    =(\mathcal{N}_{\Gamma}(\overline{x})\!+[\![\overline{v}]\!])\cap[\![d]\!]^{\perp}$.
  From $\overline{h}\in\mathcal{N}_{\mathcal{C}_{\Gamma}(\overline{x},\overline{v})}(d)$,
  there exist $\overline{\eta}\in\mathcal{N}_{\Gamma}(\overline{x})$ and
  $\overline{\alpha}\in\mathbb{R}$ such that
  $\overline{h}=\overline{\eta}+\overline{\alpha}\overline{v}$
  and $\langle\overline{\eta}+\overline{\alpha}\overline{v},d\rangle=0$. Since
  $\overline{v}\in\mathcal{N}_{\Gamma}(\overline{x})$ and
  $\overline{\eta}\in\mathcal{N}_{\Gamma}(\overline{x})$, by Corollary \ref{TNcone-Gamma},
  there exist $\overline{\lambda}\in\mathcal{N}_K(g(\overline{x}))$ and
  $\overline{\mu}\in\mathcal{N}_K(g(\overline{x}))$ such that
  $\overline{v}=\nabla g(\overline{x})\overline{\lambda}$ and
  $\overline{\eta}=\nabla g(\overline{x})\overline{\mu}$.
  Write $\overline{\xi}:=\overline{\mu}+\overline{\alpha}\overline{\lambda}$. Clearly,
  $\overline{\xi}\in\mathcal{R}_{\mathcal{N}_K(g(\overline{x}))}(\overline{\lambda})$.
  Also, from $\langle\overline{\eta}+\overline{\alpha}\overline{v},d\rangle=0$,
  we have $\langle g'(\overline{x})d,\overline{\xi}\rangle=0$.
  Together with $\overline{h}=\nabla g(\overline{x})\overline{\xi}$
  and $\overline{\lambda}\in\mathcal{M}_{\overline{x}}(\overline{v})$,
  we conclude that $\overline{h}$ belongs to the set on the right hand side of \eqref{critical-cone2}.
  Thus, $\mathcal{N}_{\mathcal{C}_{\Gamma}(\overline{x},\overline{v})}(d)$
  is included in the set on the right hand side of \eqref{critical-cone2}.
  The proof is completed.
  \end{proof}

  The sets on the right hand side of \eqref{critical-cone1} and
  \eqref{critical-cone2} are generally not closed. Proposition
  \ref{critical-normal-prop1} shows that their closedness is implied by
  that of $\mathcal{R}_{\mathcal{N}_{\Gamma}(\overline{x})}(\overline{v})$.
  A checkable condition for the latter is the strict complementarity which
  implies the calmness of $\mathcal{M}_{\overline{x}}$ by Proposition \ref{conditions-closedness}.
  Following \cite{BS00}, we say that the {\bf strict complementarity condition}
  holds for the system $g(x)\in K$ at
  $(\overline{x},\overline{v})\!\in{\rm gph}\mathcal{N}_{\Gamma}$
  if there is $\lambda\in {\rm ri}(\mathcal{N}_{K}(g(\overline{x})))$ such that
  $\overline{v}=\nabla g(\overline{x})\lambda$.
%-----------------------------------------------------------------------------------------
 \begin{proposition}\label{conditions-closedness}
  Let $(\overline{x},\overline{v})\in{\rm gph}\mathcal{N}_{\Gamma}$.
  Suppose $\mathcal{G}$ is metrically subregular at $\overline{x}$ for $0$.
  If the strict complementarity condition holds at $(\overline{x},\overline{v})$,
  then the radial cone $\mathcal{R}_{\mathcal{N}_{\Gamma}(\overline{x})}(\overline{v})$ is closed,
  and the multifunction $\mathcal{M}_{\overline{x}}$ is calm at $\overline{v}$
  for each $\lambda\in\mathcal{M}_{\overline{x}}(\overline{v})$.
 \end{proposition}
 \begin{proof}
  The first part follows by \cite[Proposition 2.1]{Gfrerer171}.
  We prove the second part.
  Notice that $\mathcal{M}_{\overline{x}}$
  can be rewritten as $\mathcal{M}_{\overline{x}}(v)
  =\{\lambda\in\mathcal{N}_{K}(g(\overline{x}))\ |\ \nabla\!g(\overline{x})\lambda=v-\overline{v}\}$
  for $v\in\mathbb{X}$. Define $\mathcal{F}(u):=\{\lambda\in\mathcal{N}_{K}(g(\overline{x}))\ |\ \nabla\!g(\overline{x})\lambda-u=0\}$ for $u\in\mathbb{X}$.
  Fix an arbitrary $\lambda\in\mathcal{M}_{\overline{x}}(\overline{v})$.
  It is easy to verify that $\mathcal{M}_{\overline{x}}$ is calm
  at $\overline{v}$ for $\lambda$ iff $\mathcal{F}$ is calm
  at the origin for $\lambda$. By \cite[Page 211-212]{Ioffe08},
  the latter is equivalent to the existence of $\delta,\gamma>0$
  such that for all $\lambda'\in\mathbb{B}(\lambda,\delta)$
  \[
    {\rm dist}\big(\lambda',\mathcal{M}_{\overline{x}}(\overline{v})\big)
    \le\gamma\max\big\{{\rm dist}(\lambda',\mathcal{N}_{K}(g(\overline{x}))),
    \|-\overline{v}+\nabla g(\overline{x})\lambda'\|\big\}.
  \]
  This metric qualification holds under the strict complementarity condition
  by the convexity of $\mathcal{N}_{K}(g(\overline{x}))$ and \cite[Corollary 3]{Bauschke99}.
 \end{proof}

 It is worthwhile to point out that the strict complementarity condition
 is not necessary for the closedness of the radial cone
 $\mathcal{R}_{\mathcal{N}_{\Gamma}(\overline{x})}(\overline{v})$;
 see the following example.
%-------------------------------------------------------------------------------
 \begin{example}\label{example1}
  Let $g(x,t):=\left(\begin{matrix}
                   {\rm Diag}(x)+tE+I\\ t
               \end{matrix}\right)$
  for $x\in\mathbb{R}^2$ and $t\in\mathbb{R}$,
  where $I$ is the $2\times 2$ identity matrix and
  $E$ is the $2\times 2$ matrix of all ones.
  Consider the constraint system $g(x,t)\in K:=\mathbb{S}^2_+\times \mathbb{R}_+$
  where $\mathbb{S}_{+}^2$ is the $2\times 2$ positive semidefinite matrix cone.
  Let
  \[
    \overline{x}=(-1,-1)^{\mathbb{T}},\ \overline{t}=0,\,\overline{\lambda}=0_{2\times 2},\,\overline{\tau}=0
    \ {\rm and}\ \overline{v}=((0,0)^{\mathbb{T}};0).
  \]
  Since $g(\overline{x},\overline{t})=(0_{2\times 2},0)$, clearly,
  $(\overline{x},\overline{t})\in g^{-1}(K):=\Gamma$
  and $\mathcal{N}_{K}(g(\overline{x},\overline{t}))=\mathbb{S}_{-}^2\times \mathbb{R}_-$.
  Since
  \begin{equation}\label{grad-gfun}
    \nabla g(\overline{x},\overline{t})(H,\omega)
    =\left(\begin{matrix}
           {\rm diag}(H)\\ \langle E,H\rangle+\omega
           \end{matrix}\right)\quad\forall (H,\omega)\in\mathbb{S}^2\times\mathbb{R},
  \end{equation}
  we have $\overline{v}=\nabla g(\overline{x},\overline{t})(\overline{\lambda},\overline{\tau})$,
  and then $\overline{v}\in\mathcal{N}_{\Gamma}(\overline{x},\overline{t})$.
  Since
  \(
    {\rm ri}(\mathcal{N}_{K}(g(\overline{x},\overline{t})))
    =\mathbb{S}_{--}^2\times \mathbb{R}_{--},
  \)
  There does not exist
  $(\lambda,\tau)\in{\rm ri}(\mathcal{N}_{K}(g(\overline{x},\overline{t})))$ such that
  $\nabla g(\overline{x},\overline{t})(\lambda,\tau)=\overline{v}$,
  but since $\overline{v}=((0,0)^{\mathbb{T}};0)$,
  the radial cone $\mathcal{R}_{\mathcal{N}_{\Gamma}(\overline{x},\overline{t})}(\overline{v})
  =\mathcal{N}_{\Gamma}(\overline{x},\overline{t})$ is closed.
 \end{example}

  Next we provide another characterization for the normal cone to the critical cone.
 %-----------------------------------------------------------------------------------------
 \begin{proposition}\label{critical-normal-prop2}
  Let $(\overline{x},\overline{v})\in{\rm gph}\mathcal{N}_{\Gamma}$.
  Suppose $\mathcal{G}$ is metrically subregular at $\overline{x}$ for $0$.
  If $\mathcal{M}_{\overline{x}}$ is isolated calm at $\overline{v}$
  for some $\lambda\in\mathcal{M}_{\overline{x}}(\overline{v})$,
  then for any given $d\in\mathcal{C}_{\Gamma}(\overline{x},\overline{v})$,
  \begin{equation}\label{critical-cone3}
  \mathcal{N}_{\mathcal{C}_{\Gamma}(\overline{x},\overline{v})}(d)
  =\Big\{\nabla g(\overline{x})\xi\ |\ \langle\xi,g'(\overline{x})d\rangle=0,\,
        \xi\in\mathcal{T}_{\mathcal{N}_{K}(g(\overline{x}))}(\lambda)\Big\}.
  \end{equation}
 \end{proposition}
 \begin{proof}
  By the first part of Proposition \ref{critical-normal-prop1},
  we only need to prove that $\mathcal{N}_{\mathcal{C}_{\Gamma}(\overline{x},\overline{v})}(d)$
  is included in the set on the right hand side of \eqref{critical-cone3}.
  Let $h^*$ be an arbitrary
  point from $\mathcal{N}_{\mathcal{C}_{\Gamma}(\overline{x},\overline{v})}(d)$.
  From \eqref{temp-equa1-sec2}, there exists a sequence
  $\{h^k\}\subseteq\mathcal{N}_{\Gamma}(\overline{x})+[\![\overline{v}]\!]$
  such that $h^k\to h^*$ with $\langle h^*,d\rangle=0$. By the expression of
  $\mathcal{N}_{\Gamma}(\overline{x})$, for each $k$ there exist
  $\lambda^k\in\mathcal{N}_{K}(g(\overline{x}))$ and $\alpha_k\in\mathbb{R}$
  such that $h^k=\nabla g(\overline{x})\lambda^k+\alpha_k\overline{v}$.
  Since $\lambda\in\mathcal{M}_{\overline{x}}(\overline{v})$,
  we have $\overline{v}=\nabla g(\overline{x})\lambda$,
  and $h^k=\nabla g(\overline{x})(\lambda^k\!+\alpha_k\lambda)$ for each $k$.
  Notice that $\{\lambda^k\!+\alpha_k\lambda\}$ is bounded.
  If not, by using
  \[
   \frac{h^k}{\|\lambda^k\!+\alpha_k\lambda\|}
   =\nabla g(\overline{x})\frac{\lambda^k\!+\alpha_k\lambda}{\|\lambda^k\!+\alpha_k\lambda\|}
   \ \ {\rm and}\ \
   \frac{\lambda^k\!+\alpha_k\lambda}{\|\lambda^k\!+\alpha_k\lambda\|}
   \in \mathcal{N}_K(g(\overline{x}))\!+[\![\lambda]\!]\subseteq
   \mathcal{T}_{\mathcal{N}_K(g(\overline{x}))}(\lambda),
  \]
  there exists $0\ne\overline{\mu}\in{\rm Ker}(\nabla g(\overline{x}))
  \cap\mathcal{T}_{\mathcal{N}_K(g(\overline{x}))}(\lambda)\ne\{0\}$.
  This, by Proposition \ref{prop-Mx}, contradicts the isolated calmness
  assumption of $\mathcal{M}_{\overline{x}}$ at $\overline{v}$ for $\lambda$.
  Now we assume (if necessary taking a subsequence) that
  $\lambda^k+\alpha_k\lambda\to\xi$. Clearly,
  $\xi\in\mathcal{T}_{\mathcal{N}_K(g(\overline{x}))}(\lambda)$
  and $h^*=\nabla g(\overline{x})\xi$. Together with $\langle h^*,d\rangle=0$,
  we have $\langle g'(\overline{x})d,\xi\rangle=0$. This shows that $h^*$ belongs to
  the set on the right hand side of \eqref{critical-cone3},
  and the claimed inclusion follows.
 \end{proof}
%----------------------------------------------------------------------------------------
 \begin{remark}\label{remark-critical}
  Consider an arbitrary $(\overline{x},\overline{v})\!\in{\rm gph}\mathcal{N}_{\Gamma}$
  and an arbitrary $d\in\mathcal{C}_{\Gamma}(\overline{x},\overline{v})$.
  Under the assumption of Proposition \ref{critical-normal-prop2},
  by using Lemma \ref{dir-proj} and $[\mathcal{C}_K(g(\overline{x}),\lambda)]^{\circ}
  =\mathcal{T}_{\mathcal{N}_K(g(\overline{x}))}(\lambda)$,
  \begin{align*}
   \nabla g(\overline{x})\Big[D\mathcal{N}_{K}(g(\overline{x})|\lambda)(g'(\overline{x})d)
    - \frac{1}{2}\nabla\Upsilon(g'(\overline{x})d)\Big]
   \!=\mathcal{N}_{\mathcal{C}_{\Gamma}(\overline{x},\overline{v})}(d)
   \!=\nabla g(\overline{x})\mathcal{N}_{\mathcal{C}_{K}(g(\overline{x}),\lambda)}(g'(\overline{x})d)
  \end{align*}
  with $\Upsilon(\cdot)=-\sigma(\lambda,\mathcal{T}_{K}^2(g(\overline{x}),\cdot))$,
  where the last equality is using the following equivalence
  \[
    \xi\in\mathcal{N}_{\mathcal{C}_K(g(\overline{x}),\overline{\lambda})}(g'(\overline{x})d)
    \Longleftrightarrow \langle \xi,g'(\overline{x})d\rangle=0,
    \xi\in\mathcal{T}_{\mathcal{N}_K(g(\overline{x}))}(\overline{\lambda}).
  \]
 \end{remark}

  It is worthwhile to point out that there is no direct relation between
  the closedness of $\mathcal{R}_{\mathcal{N}_{\Gamma}(\overline{x})}(\overline{v})$
  and the isolated calmness of $\mathcal{M}_{\overline{x}}$; see Example \ref{example2} below.
  In addition, although the strict complementarity condition and
  the isolated calmness of $\mathcal{M}_{\overline{x}}$ imply
  the calmness of $\mathcal{M}_{\overline{x}}$, there is no direct relation
  between them; see Example \ref{example3} below.
%----------------------------------------------------------------------------
 \begin{example}\label{example2}
  Consider the constraint system in Example \ref{example1}.
  Let $(\overline{x},\overline{t})$ and $(\overline{\lambda},\overline{\tau})$
  be same as Example \ref{example1}. Firstly, by using \eqref{grad-gfun} and noting that $\mathcal{T}_{\mathcal{N}_K(g(\overline{x},\overline{t}))}(\overline{\lambda},\overline{\tau})
  =\mathbb{S}_{-}^2\times \mathbb{R}_{-}$, it is not hard to check that
  \(
    {\rm Ker}(\nabla g(\overline{x},\overline{t}))\cap\mathcal{T}_{\mathcal{N}_K(g(\overline{x},\overline{t}))}(\overline{\lambda},\overline{\tau})
    =\{(0_{2\times 2},0)\}.
  \)
  By Proposition \ref{prop-Mx}, the multifunction
  $\mathcal{M}_{(\overline{x},\overline{t})}$ is isolated calm
  at $\overline{v}=((0,0)^{\mathbb{T}};0)$ for $(\overline{\lambda},\overline{\tau})$.

  \medskip

  Next we consider $(\widehat{\lambda},\overline{\tau})$ with
  $\widehat{\lambda}=\left[\begin{matrix}
              -1 & 0\\ 0 & 0
   \end{matrix}\right]\in\mathbb{S}_{-}^2$.
  By using \eqref{grad-gfun}, we calculate that
  \[
   \widehat{v}=\nabla g(\overline{x},\overline{t})(\widehat{\lambda},\overline{\tau})
  =((-1,0)^{\mathbb{T}};-1).
  \]
  Since
  \(
    \mathcal{T}_{\mathcal{N}_K(g(\overline{x},\overline{t}))}(\widehat{\lambda},\overline{\tau})
    =\mathcal{T}_{\mathbb{S}_{-}^2}(\widehat{\lambda})\times\mathcal{T}_{\mathbb{R}_{-}}(\overline{\tau})
    =\{H\in\mathbb{S}^2\ |\ H_{22}\le 0\}\times\mathbb{R}_{-},
  \)
  it follows that
  \[
    {\rm Ker}(\nabla g(\overline{x},\overline{t}))\cap\mathcal{T}_{\mathcal{N}_K(g(\overline{x},\overline{t}))}(\overline{Y},\overline{s})
    \ne \{(0_{2\times 2},0)\}.
  \]
  By Proposition \ref{prop-Mx}, the mapping $\mathcal{M}_{(\overline{x},\overline{t})}$
  is not isolated calm at $\widehat{v}$ for $(\widehat{\lambda},\overline{\tau})$.
  Notice that
  \begin{align*}
    \mathcal{R}_{\mathcal{N}_{\Gamma}(\overline{x},\overline{t})}(\overline{v})
    &=\nabla g(\overline{x},\overline{t})\big(\mathcal{N}_K(g(\overline{x},\overline{t}))+[\![(\widehat{\lambda},\overline{\tau})]\!]\big)\\
    &=\bigg\{\left(\begin{matrix}
                   {\rm diag}(Y+a\widehat{\lambda})\\  \langle E, Y\rangle +\tau
               \end{matrix}\right)\;|\; Y\in \mathbb{S}^2_-,\,\tau\in\mathbb{R}_{-},\,a\in \mathbb{R}\bigg\}.
  \end{align*}
  Clearly, $\mathcal{R}_{\mathcal{N}_{\Gamma}(\overline{x},\overline{t})}(\overline{v})
  \subseteq\mathbb{R}\times\mathbb{R}_{-}\times \mathbb{R}_{-}$. Furthermore,
  for any $(\omega;b;\pi)\in \mathbb{R}\times\mathbb{R}_{-}\times\mathbb{R}_{-}$,
  \[
  \left(\begin{matrix}
   \omega\\ b\\ \pi
   \end{matrix}\right)
   =\left(\begin{matrix}
    {\rm diag}(Y+a\widehat{\lambda})\\  \langle E, Y\rangle
     \end{matrix}\right)\ {\rm with}\
  Y=\left(\begin{matrix}
        b & -b\\ -b & b
        \end{matrix}\right)\in\mathbb{S}_{-}^2,\, a=b-\omega\in\mathbb{R},\tau=\pi\in\mathbb{R}_{-}.
  \]
 This shows that $\mathcal{R}_{\mathcal{N}_{\Gamma}(\overline{x},\overline{t})}(\overline{v})
  =\mathbb{R}\times\mathbb{R}_{-}\times \mathbb{R}_{-}$, and hence is closed,
 although $\mathcal{M}_{(\overline{x},\overline{t})}$ is not isolated calm
 at $\widehat{v}$ for $(\widehat{\lambda},\overline{\tau})$.
 Along with the arguments in the first paragraph,
 we conclude that the isolated calmness of $\mathcal{M}_{(\overline{x},\overline{t})}$
 has no relation with the closedness of $\mathcal{R}_{\mathcal{N}_{\Gamma}(\overline{x})}(\overline{v})$.
 \end{example}
 %----------------------------------------------------------------------------
 \begin{example}\label{example3}
  Consider the constraint system $g(X)\in K$
  where $K=\{0_{2\times 2}\}\times\mathbb{S}_{+}^2$ and
  \[
    g(X):=\left(\begin{matrix}
                   X+C\\ X
              \end{matrix}\right)
                   \ {\rm with}\
    C:=\left(\begin{matrix}
             0 & 0\\ 0 & -1
       \end{matrix}\right)\ \ {\rm for}\ X\in\mathbb{S}^2.
  \]
  Notice that
  \(
    \nabla g(X)(Y,S)=Y+S
  \)
  for $Y,S\in\mathbb{S}^2$. We consider the following points:
  \[
   \overline{X}=\left(\begin{matrix}
             0 & 0\\ 0 & 1
       \end{matrix}\right),\
   \overline{S}=\left(\begin{matrix}
             -1 & 0\\ 0 & 0
       \end{matrix}\right),\
   \overline{Y}=0_{2\times 2}\ \ {\rm and}\ \
   \overline{v}=\overline{S}.
  \]
 Clearly, $(\overline{Y},\overline{S})\in{\rm ri}(\mathcal{N}_{K}(g(\overline{X})))$
 and $(\overline{Y},\overline{S})\in\mathcal{M}_{\overline{X}}(\overline{v})$.
 The strict complementarity condition is satisfied at $(\overline{X},\overline{v})$, but
 $\mathcal{M}_{\overline{X}}$ is not isolated
 calm at $\overline{v}$ since $\mathcal{M}_{\overline{X}}(\overline{v})$
 is not singleton. Together with Example \ref{example1}, we conclude that
 the strict complementarity condition has no relation with
 the isolated calmness of $\mathcal{M}_{\overline{x}}$.
 \end{example}
%-------------------------------------------------------------------------------------------
 \section{Graphical derivative of the mapping $\mathcal{N}_{\Gamma}$}\label{sec3}

  By Corollary \ref{TNcone-Gamma}, when $\mathcal{G}$ is metrically subregular
  at $\overline{x}$ for $0$, $(\overline{x},\overline{v})\in{\rm gph}\mathcal{N}_{\Gamma}$
  if and only if there exists $\overline{\lambda}\in\mathcal{N}_K(g(\overline{x}))$ such that
  $\overline{v}=\nabla g(\overline{x})\overline{\lambda}$. By this, we define the mapping
  \begin{equation}\label{Phimap}
    \Phi(x,\lambda,v):=
    \left(\begin{matrix}
      -v+\nabla g(x)\lambda\\
       g(x)-\Pi_K(g(x)\!+\lambda)
   \end{matrix}\right)\quad{\rm for}\ (x,\lambda,v)\in\mathbb{X}\times\mathbb{Y}\times\mathbb{X}.
  \end{equation}
  Since $\Pi_{K}$ is directionally differentiable at $x$ in the Hadamard sense
  by \cite[Theorem 7.2]{BCS98} and \cite[Proposition 2.49]{BS00} and
  the mapping $\nabla g$ is continuously differentiable, the mapping $\Phi$
  is locally Lipschitz and directionally differentiable. In Subsection \ref{subsec3.1},
  we shall characterize the graphical derivative of $\mathcal{N}_{\Gamma}$
  under the metric subregularity of $\Phi$.
%----------------------------------------------------------------------------------------
 \subsection{Characterization for graphical derivative of $\mathcal{N}_{\Gamma}$}\label{subsec3.1}

  First we present a lower estimation for the graphical derivative
  of $\mathcal{N}_{\Gamma}$ via that of $\Phi^{-1}$.
%---------------------------------------------------------------------------------------------------
 \begin{lemma}\label{lestimate-lemma1}
  Consider an arbitrary $(\overline{x},\overline{v})\in{\rm gph}\mathcal{N}_{\Gamma}$.
  Suppose that the multifunction $\mathcal{G}$ in \eqref{MGmap} is metrically subregular
  at $\overline{x}$ for the origin, and that the mapping $\Phi$ is metrically subregular
  at each $(\overline{x},\lambda,\overline{v})$
  with $\lambda\in\mathcal{M}_{\overline{x}}(\overline{v})$ for the origin.
  Then, it holds that
  \[
   \!\mathcal{T}_{{\rm gph}\mathcal{N}_{\Gamma}}(\overline{x},\overline{v})
    \supseteq\!\bigcup_{\lambda\in\mathcal{M}_{\overline{x}}(\overline{v})}\!
    \Big\{(d,w)\in\mathbb{X}\times\mathbb{X}\ |\ \exists\mu\in\mathbb{Y}\ {\rm s.t.}\ (d,\mu,w)\in D\Phi^{-1}((0,0)|(\overline{x},\lambda,\overline{v}))(0,0)\Big\}.
  \]
 \end{lemma}
 \begin{proof}
  Define $\mathcal{A}(x,y,x'):=(x,x')$ for $(x,y,x')\in\mathbb{X}\times\mathbb{Y}\times\mathbb{X}$.
  By Remark \ref{robust-remark}, there exists a neighborhood $\mathcal{V}$ of $\overline{x}$
  such that the multifunction $\mathcal{G}$ in \eqref{MGmap} is metrically subregular
  at each $x\in\mathcal{V}\cap\Gamma$ for the origin. From Corollary \ref{TNcone-Gamma},
  it follows that
  \[
  {\rm gph}\mathcal{N}_{\Gamma}\cap(\mathcal{V}\times\mathbb{X})
  =\mathcal{A}(\Phi^{-1}(0,0))\cap(\mathcal{V}\times\mathbb{X}).
  \]
  By virtue of \cite[Theorem 6.43]{RW98}, we obtain the following inclusion
  \begin{align}\label{temp-tcone-Nomega}
   \mathcal{T}_{{\rm gph}\mathcal{N}_{\Gamma}}(\overline{x},\overline{v})
   &\supseteq\bigcup_{z\in\mathcal{A}^{-1}(\overline{x},\overline{v})\cap\Phi^{-1}(0,0)}
    \Big\{\mathcal{A}(\xi,\eta,\zeta)\ |\ (\xi,\eta,\zeta)\in\mathcal{T}_{\Phi^{-1}(0,0)}(z)\Big\}\nonumber\\
   &=\bigcup_{\lambda\in\mathcal{M}_{\overline{x}}(\overline{v})}
    \Big\{\mathcal{A}(\xi,\eta,\zeta)\ |\ (\xi,\eta,\zeta)\in\mathcal{T}_{\Phi^{-1}(0,0)}(\overline{x},\lambda,\overline{v})\Big\},
  \end{align}
  where the equality is due to the definitions of $\mathcal{A}$ and
  $\mathcal{M}_{\overline{x}}(\overline{v})$.
  Since $\Phi$ is metrically subregular at each $(\overline{x},\lambda,\overline{v})$
  with $\lambda\in\mathcal{M}_{\overline{x}}(\overline{v})$ for the origin,
  by virtue of Lemma \ref{TF-relation},
  \[
    (\xi,\eta,\zeta)\in\mathcal{T}_{\Phi^{-1}(0,0)}(\overline{x},\lambda,\overline{v})
    \Longleftrightarrow (0,0,\xi,\eta,\zeta)\in\mathcal{T}_{{\rm gph}\Phi^{-1}}(0,0,\overline{x},\lambda,\overline{v}).
  \]
  Together with the inclusion in \eqref{temp-tcone-Nomega} and the definition of $\mathcal{A}$,
  it follows that
  \[
    \mathcal{T}_{{\rm gph}\mathcal{N}_{\Gamma}}(\overline{x},\overline{v})
    \supseteq\!\bigcup_{\lambda\in\mathcal{M}_{\overline{x}}(\overline{v})}
    \Big\{(\xi,\zeta)\ |\ \exists\eta\in\mathbb{Y}\ {\rm s.t.}\
    (\xi,\eta,\zeta)\in D\Phi^{-1}((0,0)|(\overline{x},\lambda,\overline{v}))(0,0)\Big\}.
  \]
  This shows that the desired inclusion holds. The proof is completed.
 \end{proof}

 The following lemma gives the characterization on the graphical derivative of $\Phi^{-1}$.
%--------------------------------------------------------------------------------------
 \begin{lemma}\label{Phi-derivative}
  Let $\Phi$ be defined by \eqref{Phimap}. Consider an arbitrary
  $(\overline{x},\overline{\lambda},\overline{v})\in\Phi^{-1}(0,0)$. Then,
  \begin{align*}
   &(\Delta x,\Delta\lambda,\Delta v)
     \in D\Phi^{-1}((0,0)|(\overline{x},\overline{\lambda},\overline{v}))(\Delta \xi,\Delta\eta)\\
   &\Longleftrightarrow
    \left\{\begin{array}{ll}
     \Delta\xi=\nabla^2\langle\overline{\lambda},g\rangle(\overline{x})\Delta x+\nabla g(\overline{x})\Delta\lambda-\Delta v;\\
     \Delta\eta=g'(\overline{x})\Delta x-\Pi_K'(g(\overline{x})\!+\overline{\lambda};g'(\overline{x})\Delta x\!+\!\Delta\lambda).
      \end{array}\right.
  \end{align*}
 \end{lemma}
 \begin{proof}
  Since the mapping $\Phi$ is locally Lipschitz and directionally differentiable,
  we have
  \begin{align*}
   &D\Phi((\overline{x},\overline{\lambda},\overline{v})|(0,0))(\Delta x,\Delta\lambda,\Delta v)\\
   &=\Big\{(\Delta \xi,\Delta\eta)\in\mathbb{X}\times\mathbb{Y}\ |\
      \Phi'((\overline{x},\overline{\lambda},\overline{v});(\Delta x,\Delta\lambda,\Delta v))
      =(\Delta \xi,\Delta\eta)\Big\}.
  \end{align*}
  In addition, by the expression of $\Phi$ and \cite[Proposition 2.47]{BS00},
  we calculate that
  \[
    \Phi'((\overline{x},\overline{\lambda},\overline{v});(\Delta x,\Delta\lambda,\Delta v))
    =\left(\begin{matrix}
       \nabla^2\langle\overline{\lambda},g\rangle(\overline{x})\Delta x+\nabla g(\overline{x})\Delta\lambda-\Delta v\\
       g'(\overline{x})\Delta x-\Pi_K'(g(\overline{x})\!+\overline{\lambda};g'(\overline{x})\Delta x\!+\Delta\lambda)
     \end{matrix}\right).
  \]
  Notice that $(\Delta x,\Delta\lambda,\Delta v)
     \in D\Phi^{-1}((0,0)|(\overline{x},\overline{\lambda},\overline{v}))(\Delta \xi,\Delta\eta)$
  if and only if $(\Delta \xi,\Delta\eta)$ lies in
  $D\Phi((\overline{x},\overline{\lambda},\overline{v})|(0,0))(\Delta x,\Delta\lambda,\Delta v)$.
  The result follows from the last two equations.
  \end{proof}

  By combining Lemma \ref{lestimate-lemma1} with Lemma \ref{Phi-derivative}
  and using Lemma \ref{dir-proj}, we readily obtain a lower estimation for
  the graphical derivative of the mapping $\mathcal{N}_{\Gamma}$.
%---------------------------------------------------------------------------------------------------
 \begin{proposition}\label{lestimate}
  Consider an arbitrary $(\overline{x},\overline{v})\in{\rm gph}\mathcal{N}_{\Gamma}$.
  Suppose that the multifunction $\mathcal{G}$ in \eqref{MGmap} is metrically subregular
  at $\overline{x}$ for the origin. If the mapping $\Phi$ in \eqref{Phimap}
  is metrically subregular at each $(\overline{x},\lambda,\overline{v})$ with
  $\lambda\in\mathcal{M}_{\overline{x}}(\overline{v})$ for the origin,
  then
  \[
   \!\mathcal{T}_{{\rm gph}\mathcal{N}_{\Gamma}}(\overline{x},\overline{v})
   \supseteq\!\bigcup_{\lambda\in\mathcal{M}_{\overline{x}}(\overline{v})}\!
    \Big\{(d,w)\in\mathbb{X}\times\mathbb{X}\ |\ w\in \nabla^2\langle\lambda,g\rangle(\overline{x})d+\nabla\!g(\overline{x})
     D\mathcal{N}_K(g(\overline{x})|\lambda)(g'(\overline{x})d)\Big\}.
  \]
 \end{proposition}
 \begin{remark}\label{low-estimate-remark}
  During the reviewing of this paper, we learned that Gfrerer and Mordukhovich
  only under the metric subregularity of $\mathcal{G}$ derived a lower estimation
  for the graphical derivative of $\mathcal{N}_{\Gamma}$ (see \cite[Theorem 3.3]{Gfrerer171}),
  which has a little difference from the one
  in Proposition \ref{lestimate} but agrees with it
  under the closedness of $\mathcal{R}_{\mathcal{N}_{\Gamma}(\overline{x})}(\overline{v})$.
 \end{remark}

  Next we concentrate on an upper estimation for the graphical derivative
  of $\mathcal{N}_{\Gamma}$.
 %-------------------------------------------------------------------------------------
  \begin{proposition}\label{uestimate}
   Consider an arbitrary $(\overline{x},\overline{v})\in{\rm gph}\mathcal{N}_{\Gamma}$.
   Suppose that $\mathcal{G}$ is metrically subregular at $\overline{x}$ for $0$,
   and that $\mathcal{M}_{\overline{x}}$ is isolated calm at $\overline{v}$
   for some $\overline{\lambda}\in\mathcal{M}_{\overline{x}}(\overline{v})$. Then,
   \begin{equation*}%\label{Tcone-Nomega-equa}
    \!\mathcal{T}_{{\rm gph}\mathcal{N}_{\Gamma}}(\overline{x},\overline{v})
   \subseteq\Big\{(d,w)\in\mathbb{X}\times\mathbb{X}\ |\ w\in \nabla^2\langle\overline{\lambda},g\rangle(\overline{x})d+\nabla\!g(\overline{x})
     D\mathcal{N}_K(g(\overline{x})|\overline{\lambda})(g'(\overline{x})d)\Big\}.
  \end{equation*}
  \end{proposition}
  \begin{proof}
  Since $\mathcal{G}$ is metrically subregular at $\overline{x}$ for the origin
  and $(\overline{x},\overline{v})\in{\rm gph}\,\mathcal{N}_{\Gamma}$,
  by Corollary \ref{TNcone-Gamma}, $\mathcal{M}_{\overline{x}}(\overline{v})\ne\emptyset$.
  Since $\mathcal{M}_{\overline{x}}$ is isolated calm at $\overline{v}$ for
  $\overline{\lambda}\in\mathcal{M}_{\overline{x}}(\overline{v})$,
  by Proposition \ref{prop-Mx} the SRCQ for the system $g(x)\in K$ holds
  at $\overline{x}$ w.r.t. $\overline{\lambda}$. So,
  $\mathcal{M}_{\overline{x}}(\overline{v})=\{\overline{\lambda}\}$
  and Robinson's CQ for this system holds at $\overline{x}$.
  Now fix an arbitrary $(d,w)\in\mathcal{T}_{{\rm gph}\widehat{\mathcal{N}}_{\Gamma}}(\overline{x},\overline{v})$.
  Then, there exist $t_k\downarrow 0$ and $(d^k,w^k)\to(d,w)$ such that
  $(\overline{x}+t_kd^k,\overline{v}+t_kw^k)\in{\rm gph}\widehat{\mathcal{N}}_{\Gamma}$ for each $k$.
  Write $x^k:=\overline{x}+t_kd^k$ and $v^k:=\overline{v}+t_kw^k$.
  Since Robinson's CQ for the system $g(x)\in K$ holds at $\overline{x}$,
  there exists a neighborhood $\mathcal{U}$ of $\overline{x}$ such that
  Robinson's CQ for this system holds at each $z\in\mathcal{U}$.
  By Corollary \ref{TNcone-Gamma}, for each sufficiently large $k$,
  there exists $\lambda^k\in\mathcal{N}_K(g(x^k))$ such that
  $v^k=\nabla g(x^k)\lambda^k$. Furthermore, the sequence $\{\lambda^k\}$ is bounded.
  Taking a subsequence if necessary, we assume that $\{\lambda^k\}$ converges
  to $\widehat{\lambda}$. Since $\lambda^k\in\mathcal{N}_{K}(g(x^k))$,
  from the outer semicontinuity of $\mathcal{N}_K$ it follows that
  $\widehat{\lambda}\in\mathcal{N}_K(g(\overline{x}))$. In addition,
  from $v^k=\nabla g(x^k)\lambda^k$ we have $\overline{v}=\nabla g(\overline{x})\widehat{\lambda}$.
  This means that $\widehat{\lambda}\in\mathcal{M}_{\overline{x}}(\overline{v})=\{\overline{\lambda}\}$.

  \medskip

  By Theorem \ref{NK-calm}, $\mathcal{N}_K$ is calm at $g(\overline{x})$
  for $\overline{\lambda}$, i.e., there exist $\delta>0$ and $c>0$ such that
  \[
    \mathcal{N}_K(y)\cap\mathbb{B}(\overline{\lambda},\delta)
    \subset\mathcal{N}_K(g(\overline{x}))+c\|y-g(\overline{x})\|\mathbb{B}_{\mathbb{Y}}
    \quad\ \forall y\in\mathbb{Y}.
  \]
  From $\mathcal{N}_K(g(x^k))\ni\lambda^k\to\overline{\lambda}$,
  for each $k$ large enough, there exists $\zeta^k\in\mathcal{N}_K(g(\overline{x}))$ satisfying
  \begin{equation}\label{lambdak-ineq1}
    \|\lambda^k-\zeta^k\|={\rm dist}(\lambda^k,\mathcal{N}_K(g(\overline{x})))
    \le c\|g(x^k)-g(\overline{x})\|
    =ct_k\|g'(\overline{x})d^k+o(t_k)/t_k\|
  \end{equation}
  where the second equality is by the Taylor expansion of $g(x^k)$ at $\overline{x}$.
  Write $\widetilde{v}^k:=\nabla g(\overline{x})\zeta^k$. Clearly,
  $\zeta^k\in\mathcal{M}_{\overline{x}}(\widetilde{v}^k)$.
  Also, the last inequality implies $\zeta^k\to\overline{\lambda}$.
  By the isolated calmness of $\mathcal{M}_{\overline{x}}$ at $\overline{v}$
  for $\overline{\lambda}$, there exists a constant $\gamma>0$ (depending on
  $\overline{\lambda}$ and $\overline{v}$ only) such that for each $k$ large enough, $\|\zeta^k\!-\!\overline{\lambda}\|\le\gamma\|\overline{v}-\widetilde{v}^k\|$.
  Notice that
  \[
    \widetilde{v}^k=\overline{v}+t_kw^k+(\nabla g(\overline{x})-\nabla g(x^k))\zeta^k
    +\nabla g(x^k)(\zeta^k-\lambda^k).
  \]
  By virtue of $\|\nabla g(x^k)-\nabla g(\overline{x})\|\le t_k\|D^2g(\overline{x})d^k+o(t_k)/t_k\|$
  and \eqref{lambdak-ineq1}, we have
  \[
   \|\widetilde{v}^k-\overline{v}\|\le t_k\big[\|w^k\|+\|D^2g(\overline{x})d^k\|\|\zeta^k\|
    +c\|\nabla g(x^k)\|\|g'(\overline{x})d^k\|\big]+o(t_k)
  \]
  where $D^2g(\overline{x})$ is the second-order derivative of $g$ at $\overline{x}$.
  Along with $\|\zeta^k\!-\!\overline{\lambda}\|\le \gamma\|\overline{v}-\widetilde{v}^k\|$,
  \begin{equation}\label{wlambdak-ineq1}
    \|\zeta^k\!-\!\overline{\lambda}\|\le \gamma t_k\big[\|w^k\|+\|D^2g(\overline{x})d^k\|\|\zeta^k\|
    +c\|\nabla g(x^k)\|\|g'(\overline{x})d^k\|\big]+o(t_k).
  \end{equation}
  Write $\mu^k:=\frac{\lambda^k-\overline{\lambda}}{t_k}$.
  From inequalities \eqref{lambdak-ineq1} and \eqref{wlambdak-ineq1},
  the sequence $\{\mu^k\}$ is bounded. Taking a subsequence if necessary,
  we assume that $\mu^k$ converges to $\mu$. Notice that
  \begin{align*}
    \overline{v}+t_kw^k
    &=\nabla g(x^k)\lambda^k=\nabla g(x^k)\overline{\lambda}+\nabla g(x^k)(\lambda^k-\overline{\lambda})\\
    &=(\nabla g(\overline{x})+t_kD^2g(\overline{x})d^k)\overline{\lambda}
     +(\nabla g(\overline{x})+t_kD^2g(\overline{x})d^k)(\lambda^k-\overline{\lambda})+o(t_k)\\
    &=\overline{v}+t_k\big[\nabla^2\langle\overline{\lambda},g\rangle(\overline{x})d^k
      +\nabla g(\overline{x})\mu^k+t_k\nabla^2\langle\mu^k,g\rangle(\overline{x})d^k+o(t_k)/t_k\big].
  \end{align*}
  Hence,
  \(
    w^k=\nabla^2\langle\overline{\lambda},g\rangle(\overline{x})d^k
      +\nabla g(\overline{x})\mu^k+t_k\nabla^2\langle\mu^k,g\rangle(\overline{x})d^k+o(t_k)/t_k.
  \)
  Taking the limit, we obtain $w=\nabla^2\langle \overline{\lambda},g\rangle(\overline{x})\xi+\nabla g(\overline{x})\mu$.
  Finally, we prove that $\mu\in D\mathcal{N}_K(g(\overline{x})|\lambda)(g'(\overline{x})d)$,
  and the desired inclusion follows by the arbitrariness of
  $(d,w)\in\mathcal{T}_{{\rm gph}\mathcal{N}_{\Gamma}}(\overline{x},\overline{v})$.
  From $\lambda^k\in\mathcal{N}_K(g(x^k))$ and the first order
  expansion of $g$ at $\overline{x}$, it holds that
  \[
    \overline{\lambda}+t_k\mu^k=\lambda^k\in\mathcal{N}_K(g(\overline{x})+t_k(g'(\overline{x})d^k+o(t_k)/t_k)).
  \]
  That is, $(g(\overline{x})+t_k(g'(\overline{x})d^k+o(t_k)/t_k),\overline{\lambda}+t_k\mu^k)\in{\rm gph}\mathcal{N}_K$.
  Along with $(g(\overline{x}),\overline{\lambda})\in{\rm gph}\mathcal{N}_K$,
  we have $(g'(\overline{x})d,\mu)\in\mathcal{T}_{{\rm gph}\mathcal{N}_K}(g(\overline{x}),\lambda)$
  or equivalently $\mu\in D\mathcal{N}_K(g(\overline{x})|\lambda)(g'(\overline{x})d)$.
 \end{proof}

  From Proposition \ref{lestimate} and \ref{uestimate},
  we get the following characterization for the graphical
  derivative of the mapping $\mathcal{N}_{\Gamma}$
  without requiring the nondegeneracy of $\overline{x}$ as in
  \cite{Gfrerer17,Mordu151}.
 %-------------------------------------------------------------------------------------
  \begin{theorem}\label{festimate}
   Consider an arbitrary $(\overline{x},\overline{v})\in{\rm gph}\mathcal{N}_{\Gamma}$.
   Suppose that $\mathcal{G}$ is metrically subregular at $\overline{x}$ for the origin.
   If $\mathcal{M}_{\overline{x}}$ is isolated calm at $\overline{v}$
   for some $\overline{\lambda}\in\mathcal{M}_{\overline{x}}(\overline{v})$, then
   \begin{equation*}
   \!\mathcal{T}_{{\rm gph}\mathcal{N}_{\Gamma}}(\overline{x},\overline{v})
   \subseteq\Big\{(d,w)\in\mathbb{X}\times\mathbb{X}\ |\ w\in \nabla^2\langle\overline{\lambda},g\rangle(\overline{x})d+\nabla\!g(\overline{x})
     D\mathcal{N}_K(g(\overline{x})|\overline{\lambda})(g'(\overline{x})d)\Big\}.
   \end{equation*}
   If, in addition, the mapping $\Phi$ is metrically subregular at
   $(\overline{x},\overline{\lambda},\overline{v})$ for the origin, then
   \begin{equation}\label{Tcone-final1}
   \!\mathcal{T}_{{\rm gph}\mathcal{N}_{\Gamma}}(\overline{x},\overline{v})
   =\Big\{(d,w)\in\mathbb{X}\times\mathbb{X}\ |\ w\in \nabla^2\langle\overline{\lambda},g\rangle(\overline{x})d+\nabla\!g(\overline{x})
     D\mathcal{N}_K(g(\overline{x})|\overline{\lambda})(g'(\overline{x})d)\Big\}.
   \end{equation}
  \end{theorem}

  By combining Theorem \ref{festimate} and Remark \ref{remark-critical},
  we also have the following conclusion.
%-------------------------------------------------------------------------------------
  \begin{corollary}\label{estimate-cor}
   Consider an arbitrary $(\overline{x},\overline{v})\in{\rm gph}\mathcal{N}_{\Gamma}$.
   Suppose that $\mathcal{G}$ is metrically subregular at $\overline{x}$ for the origin.
   If $\mathcal{M}_{\overline{x}}$ is isolated calm at $\overline{v}$
   for some $\overline{\lambda}\in\mathcal{M}_{\overline{x}}(\overline{v})$ and
   the mapping $\Phi$ is metrically subregular at $(\overline{x},\overline{\lambda},\overline{v})$
   for the origin, then it holds that
   \begin{align}\label{Tcone-final2}
   \!\mathcal{T}_{{\rm gph}\mathcal{N}_{\Gamma}}(\overline{x},\overline{v})
    &=\left\{(d,w)\in\mathbb{X}\times\mathbb{X}\ \Big|
    \left.\begin{array}{ll}
     w\in\!\nabla^2\langle \overline{\lambda},g\rangle(\overline{x})d
         +\frac{1}{2}\nabla g(\overline{x})\nabla\Upsilon(g'(\overline{x})d)\!\\
         \qquad +\nabla g(\overline{x})\mathcal{N}_{\mathcal{C}_{K}(g(\overline{x}),\overline{\lambda})}(g'(\overline{x})d)
     \end{array}\right.\!\right\}.
  \end{align}
  \end{corollary}
%-----------------------------------------------------------------------------------------
 \begin{remark}\label{remark-main}
  {\bf(a)} The expression of the graphical derivative in \eqref{Tcone-final1} is
  same as the one derived in \cite[Theorem 2]{Gfrerer16-MOR}, but compared with
  that of \cite[Theorem 5.2]{Mordu151} an additional term
  $\frac{1}{2}\nabla g(\overline{x})\nabla\Upsilon(g'(\overline{x})d)$
  appears since the PDC is not imposed on $K$. Compared with
  the one in \cite[Corollary 5.4]{Gfrerer171}, unless the uniqueness of
  the multiplier set and the closedness of
  $\mathcal{R}_{\mathcal{N}_{\Gamma}(\overline{x})}(\overline{v})$
  are required there, our formula \eqref{Tcone-final1} or \eqref{Tcone-final2}
  is convenient for use.

  \medskip
  \noindent
  {\bf(b)} By Remark \ref{low-estimate-remark}, we know that
  \cite[Theorem 3.3]{Gfrerer171} and Proposition \ref{critical-normal-prop2}
  imply that the equality \eqref{Tcone-final1} or \eqref{Tcone-final2}
  actually holds without the metric subregularity of $\Phi$.
 \end{remark}

%--------------------------------------------------------------------------------------
 \subsection{Conditions for metric subregularity of $\Phi$}\label{subsec3.2}

  As pointed out in Remark \ref{remark-main}(b), due to \cite[Theorem 3.3]{Gfrerer171},
  the exact characterization of the graphical derivative of $\mathcal{N}_{\Gamma}$
  in formula \eqref{Tcone-final1} or \eqref{Tcone-final2} does not
  require the metric subregularity of $\Phi$, but we think that
  it has a separate value. So, in this part we focus on
  the metric subregularity of $\Phi$. When $K$ and $g$
  are both polyhedral, from the crucial result due to Robinson \cite{Robinson81},
  the metric subregularity of $\Phi$ automatically holds.
  When either $K$ or $g$ is non-polyhedral, the metric
  subregularity of $\Phi$ at $(\overline{x},\overline{\lambda},\overline{v})$
  for the origin is implied by the isolated calmness of $\Phi^{-1}$
  at the origin for $(\overline{x},\overline{\lambda},\overline{v})$ or by
  the Aubin property of $\Phi^{-1}$. By Proposition \ref{Phi-derivative} and
  Lemma \ref{chara-icalm}, the former is equivalent to requiring
  \begin{equation}\label{implication}
   \left\{\begin{array}{ll}
   \nabla^2\langle\overline{\lambda},g\rangle(\overline{x})\Delta x
      +\nabla g(\overline{x})\Delta\lambda-\!\Delta v=0;\\
   \Delta\lambda\in D\mathcal{N}_K(g(\overline{x})|\overline{\lambda})(g'(\overline{x})\Delta x)
   \end{array}\right.\Longrightarrow (\Delta x,\Delta\lambda,\Delta v)=(0,0,0),
  \end{equation}
  which is almost impossible due to the free $\Delta v$. We next focus on
  the latter. It is a little surprising to us that the Aubin property
  of $\Phi^{-1}$ is equivalent to the nondegeneracy.

%----------------------------------------------------------------------------------------------
  \begin{proposition}\label{property-MG}
   Consider an arbitrary $(\overline{x},\overline{v})\!\in{\rm gph}\mathcal{N}_{\Gamma}$
   with $\mathcal{M}_{\overline{x}}(\overline{v})\!\ne\emptyset$.
   Let $\overline{\lambda}\in\!\mathcal{M}_{\overline{x}}(\overline{v})$.
   The multifunction $\Phi^{-1}$ has the Aubin property at the origin
   for $(\overline{x},\overline{\lambda},\overline{v})$ if and only if
   \begin{equation}\label{coderiv-imply}
   {\rm Ker}(\nabla g(\overline{x}))\cap D^*\mathcal{N}_{K}(g(\overline{x})|\overline{\lambda})(0)=\{0\}.
   \end{equation}
   In particular, condition \eqref{coderiv-imply} is equivalent to
   the nondegeneracy of $\overline{x}$ w.r.t. the set $K$ and
   the mapping $\Xi$, where $\Xi$ is same as the one in Lemma \ref{lemma-reduction}.
  \end{proposition}
  \begin{proof}
  We first characterize the coderivative of $\Phi$ at $(\overline{x},\overline{\lambda},\overline{v})$.
  Notice that
  \[
    \Phi(x,\lambda,v)=\left(\begin{matrix}
                      \Phi_1(x,\lambda,v)\\ \Phi_2(x,\lambda,v)
                      \end{matrix}\right)\ \ {\rm with}\
    \left\{\begin{array}{ll}
     \Phi_1(x,\lambda,v):=-v+\nabla g(x)\lambda;\\
     \Phi_2(x,\lambda,v):=g(x)-\Pi_K(g(x)+\lambda).
     \end{array}\right.
  \]
  Fix an arbitrary $(\Delta \xi,\Delta\eta)\in\mathbb{X}\times\mathbb{Y}$.
  By using Lemma \ref{calculus-rule} in Appendix, we calculate that
  \[
   D^*\Phi(\overline{x},\overline{\lambda},\overline{v})(\Delta \xi,\Delta\eta)
   =\left[\begin{matrix}
           \nabla^2\langle\overline{\lambda},g\rangle(\overline{x})\Delta\xi\\
            g'(\overline{x})\Delta\xi\\
            -\Delta\xi
       \end{matrix}\right]
       +D^*\Phi_2(\overline{x},\overline{\lambda},\overline{v})(\Delta\eta).
  \]
  From the definition of $\Phi_2(x,\lambda,v)$ and \cite[Theorem 1.62]{Mordu06},
  it follows that
  \[
    D^*\Phi_2(\overline{x},\overline{\lambda},\overline{v})(\Delta\eta)
    =\left(\begin{matrix}
      \nabla\!g(\overline{x})\Delta\eta\\
       0\\ 0
       \end{matrix}\right)
    +D^*(-\Pi_K\circ h)(\overline{x},\overline{\lambda},\overline{v})\Delta\eta
  \]
  where $h(x,\lambda,v):=g(x)+\lambda$ for $(x,\lambda,v)\in\mathbb{X}\times\mathbb{Y}\times\mathbb{X}$.
  Notice that $h'(\overline{x},\overline{\lambda},\overline{v})\!:
  \mathbb{X}\times\mathbb{Y}\times\mathbb{X}\to\mathbb{Y}$ is surjective.
  By applying \cite[Theorem 1.66]{Mordu06}, we obtain
  \begin{equation*}
    D^*(\Pi_K\circ h)(\overline{x},\overline{\lambda},\overline{v})
    =\left(\begin{matrix}
      \nabla\!g(\overline{x})\\
       I\\ 0
       \end{matrix}\right)D^*\Pi_K(g(\overline{x})\!+\!\overline{\lambda}).
  \end{equation*}
  In addition, it is easy to check that $(\Delta x,\Delta\lambda,\Delta v)\in
  D^*(-\Pi_K\circ h)(\overline{x},\overline{\lambda},\overline{v})(\Delta \eta)$
  if and only if $(\Delta x,\Delta\lambda,\Delta v)\in
  D^*(\Pi_K\circ h)(\overline{x},\overline{\lambda},\overline{v})(-\Delta \eta)$.
  Together with the last three equations,
  \[
   D^*\Phi(\overline{x},\overline{\lambda},\overline{v})(\Delta \xi,\Delta\eta)
   \!=\!\left[\begin{matrix}
           \nabla^2\langle\overline{\lambda},g\rangle(\overline{x})\Delta\xi
           \!+\!\nabla\!g(\overline{x})\Delta\eta\\
            g'(\overline{x})\Delta\xi\\
            -\Delta\xi
       \end{matrix}\right]
       +\!\left(\begin{matrix}
      \nabla\!g(\overline{x})\\
       I\\ 0
       \end{matrix}\right)\!D^*\Pi_K(g(\overline{x})\!+\!\overline{\lambda})(-\Delta\eta).
  \]
  So, $(\Delta x,\Delta\lambda,\Delta v,\Delta \xi,\Delta\eta)
  \in\!\mathcal{N}_{{\rm gph}\Phi}(\overline{x},\overline{\lambda},\overline{v},0,0)$
  iff $\exists\Delta\zeta\in D^*\Pi_{K}(g(\overline{x})\!+\!\overline{\lambda})(\Delta\eta)$ such that
  \begin{equation*}\label{temp-system}
   \left\{\begin{array}{ll}
     \Delta x+\nabla^2\langle\overline{\lambda},g\rangle(\overline{x})\Delta\xi+\nabla\!g(\overline{x})\Delta\eta
     =\nabla g(\overline{x})\Delta\zeta,\\
     \Delta\lambda+g'(\overline{x})\Delta\xi=\Delta\zeta,\,\Delta v=\Delta\xi.
    \end{array}\right.
  \end{equation*}
  Consequently, $(\Delta \xi,\Delta\eta)\in D^*\Phi^{-1}((0,0)|(\overline{x},\overline{\lambda},\overline{v}))
  (0,0,0)$ if and only if $(\Delta \xi,\Delta\eta)$ satisfies
  \begin{equation*}
    \left\{\begin{array}{ll}
     \Delta\xi=0,\,\nabla\!g(\overline{x})\Delta\eta=0,\\
     0\in D^*\Pi_{K}(g(\overline{x})\!+\!\overline{\lambda})(\Delta\eta).
     \end{array}\right.
  \end{equation*}
  By Lemma \ref{chara-Aubin}, $\Phi^{-1}$ has the Aubin property at the origin
  for $(\overline{x},\overline{\lambda},\overline{v})$ if and only if
  \begin{equation}\label{temp-imply}
    \left\{\begin{array}{ll}
     \nabla\!g(\overline{x})\Delta\eta=0,\\
     0\in D^*\Pi_{K}(g(\overline{x})\!+\!\overline{\lambda})(\Delta\eta)
     \end{array}\right.
     \Longrightarrow \Delta\eta=0.
  \end{equation}
  From \cite[Exercise 6.7]{RW98} and the definition of coderivative,
  for any $(u',v')\in\mathbb{Y}\times\mathbb{Y}$,
  \begin{equation}\label{Proj-normal}
   u'\in D^*\mathcal{N}_{K}(g(\overline{x})|\overline{\lambda})(v')
   \Longleftrightarrow
   -v'\in D^*\Pi_{K}(g(\overline{x})\!+\!\overline{\lambda})(-u'\!-v').
  \end{equation}
  This show that the implication in \eqref{temp-imply} can be equivalently written as
  the one in \eqref{coderiv-imply}.

  \medskip

  Now we pay our attention to the second part.
  Let $\overline{x}$ be a nondegenerate point of $g$ w.r.t $K$ and $\Xi$.
  From \cite[Definition 4.70]{BS00},
  \(
    g'(\overline{x})\mathbb{X}+{\rm Ker}\big[\Xi'(g(\overline{x}))\big]=\mathbb{Y},
  \)
  or equivalently
  \begin{equation}\label{Nondegenerate}
    {\rm Ker}(\nabla g(\overline{x}))\cap{\rm Range}(\nabla\Xi(g(\overline{x})))=\{0\}.
  \end{equation}
  Fix an arbitrary $\Delta u\in{\rm Ker}(\nabla g(\overline{x}))
  \cap D^*\mathcal{N}_{K}(g(\overline{x})|\overline{\lambda})(0)$.
  Since $\overline{\lambda}\in\mathcal{N}_K(g(\overline{x}))$,
  by the reducibility assumption for $K$ and Lemma \ref{lemma-reduction},
  there exists a unique $\overline{\mu}\in\mathcal{N}_{D}(\Xi(g(\overline{x})))$ such that
  \(
    \overline{\lambda}=\nabla\Xi(g(\overline{x}))\overline{\mu}.
  \)
  In addition, from $\Delta u\in D^*\mathcal{N}_{K}(g(\overline{x})|\overline{\lambda})(0)$
  and \cite[Theorem 3.4]{Mordu01} with $\psi=\delta_D(\cdot)$ and $h(\cdot)=\Xi(\cdot)$,
  there exists $\Delta\mu\in D^*\mathcal{N}_{D}(\Xi(g(\overline{x}))|\overline{\lambda})(0)$
  such that
  \[
    \Delta u=\nabla\Xi(g(\overline{x}))\Delta\mu.
  \]
  Along with $\nabla\!g(\overline{x})\Delta u=0$, we get
  $\nabla\!g(\overline{x})\nabla\Xi(g(\overline{x}))\Delta\mu=0$,
  which is equivalent to saying
  \[
    \nabla\Xi(g(\overline{x}))\Delta\mu\in{\rm Ker}(\nabla g(\overline{x}))
    \cap{\rm Range}(\nabla\Xi(g(\overline{x}))).
  \]
  From equation \eqref{Nondegenerate}, it follows that $\nabla\Xi(g(\overline{x}))\Delta\mu=0$.
  By the surjectivity of $\Xi'(g(\overline{x}))$, we get $\Delta\mu=0$.
  Consequently, $\Delta u=0$, and condition \eqref{coderiv-imply} is satisfied.
  Conversely, assume that $\Phi^{-1}$ has the Aubin property at
  $(\overline{x},\overline{\lambda},\overline{v})$ for the origin.
  Notice that $\Phi^{-1}$ is exactly
  \[
    \Sigma(a,b):=\Big\{(x,\lambda,v)\in\mathbb{X}\times\mathbb{Y}\times\mathbb{X}\ |\ \Phi(x,\lambda,v)=(a,b)\Big\}.
  \]
  By following the same arguments as those for \cite[Theorem 1]{KK13},
  $\overline{x}$ is nondegenerate.
 \end{proof}

  Motivated by the recent work \cite{Gfrerer11,Gfrerer13,Gfrerer16}
  for the metric subregularity, we next provide a condition
  for the metric subregularity of $\Phi$ by means of the directional
  limiting coderivative of $\mathcal{N}_K$. In order to achieve this goal,
  we need the following lemma.
%-------------------------------------------------------------------------
 \begin{lemma}\label{WtKKT}
  Let $\widetilde{\Phi}\!:\mathbb{X}\times\mathbb{Y}\times\mathbb{X}\rightrightarrows
  \mathbb{X}\times\mathbb{Y}\times\mathbb{X}$ be the multifunction defined as follows:
  \begin{equation}
   \widetilde{\Phi}(x,\lambda,v)
    :=\left(\begin{matrix}
        -v+\nabla g(x)\lambda\\
           g(x)\\
           \lambda\\
      \end{matrix}\right)
      -\left(\begin{matrix}
        \{0\}\\
         {\rm gph}\mathcal{N}_K
      \end{matrix}\right).
  \end{equation}
  Consider an arbitrary $(\overline{x},\overline{\lambda},\overline{v})\in\Phi^{-1}(0,0)$.
  Then, $\widetilde{\Phi}$ is metrically subregular at $(\overline{x},\overline{\lambda},\overline{v})$
  for the origin if and only if $\Phi$ is metrically subregular at
  $(\overline{x},\overline{\lambda},\overline{v})$ for the origin.
 \end{lemma}
 \begin{proof}
  Suppose that the mapping $\Phi$ is metrically subregular
  at $(\overline{x},\overline{\lambda},\overline{v})$ for the origin.
  Then, there exist $\varepsilon>0$ and $\kappa>0$ such that for all $(x,\lambda,v)\in\mathbb{B}((\overline{x},\overline{\lambda},\overline{v}),\varepsilon)$,
  \[
    {\rm dist}((x,\lambda,v),\Phi^{-1}(0,0))\le\kappa\|\Phi(x,\lambda,v)\|.
  \]
  To establish the metric subregularity of $\widetilde{\Phi}$ at
  $(\overline{x},\overline{\lambda},\overline{v})$ for the origin, it suffices to
  argue that there exist $\varepsilon'>0,\delta'>0$ and $\kappa'>0$ such that
  for all $(x,\lambda,v)\in\mathbb{B}((\overline{x},\overline{\lambda},\overline{v}),\varepsilon')$,
   \begin{equation}\label{aim-ineq}
   {\rm dist}((x,\lambda,v),\widetilde{\Phi}^{-1}(0,0,0))
   \le\kappa'{\rm dist}((0,0,0),\widetilde{\Phi}(x,\lambda,v)\cap\mathbb{B}((0,0,0),\delta')).
  \end{equation}
  Set $\varepsilon'=\frac{\varepsilon}{2}$ and $\delta'=\frac{\varepsilon}{2}$.
  Fix an arbitrary $(x,\lambda,v)\in\mathbb{B}((\overline{x},\overline{\lambda},\overline{v}),\varepsilon')$.
  It suffices to consider $\widetilde{\Phi}(x,\lambda,v)\cap\mathbb{B}((0,0,0),\delta')\ne\emptyset$.
  Let $(\xi,\eta,\zeta)\in\widetilde{\Phi}(x,\lambda,v)\cap\mathbb{B}((0,0,0),\delta')$ be such that
  \begin{equation}\label{aim-equa31}
    {\rm dist}((0,0,0),\widetilde{\Phi}(x,\lambda,v)\cap\mathbb{B}((0,0,0),\delta'))
    =\|(\xi,\eta,\zeta)\|.
  \end{equation}
  From $(\xi,\eta,\zeta)\in\widetilde{\Phi}(x,\lambda,v)\cap\mathbb{B}((0,0,0),\delta')$,
  it follows that $(\xi',\eta)=\Phi(x,\lambda',v)$ with $\xi'=\xi-\nabla g(x)(\eta+\zeta)$
  and $\lambda'=\lambda-\eta-\zeta$, and moreover,
  $\|(x,\lambda',v)-(\overline{x},\overline{\lambda},\overline{v})\|\le\varepsilon$.
  By the continuity of $\nabla g$, there exists $\gamma>0$ such that
  for all $x\in\mathbb{B}(\overline{x},\varepsilon')$, $\|\nabla g(x)\|\le\gamma$.
  Then,
  \begin{align*}
   &{\rm dist}((x,\lambda,v),\widetilde{\Phi}^{-1}(0,0,0))
   ={\rm dist}((x,\lambda,v),\Phi^{-1}(0,0))\\
    &\le{\rm dist}((x,\lambda',v),\Phi^{-1}(0,0))+\|\lambda-\lambda'\|\\
    &\le\kappa{\rm dist}((0,0),\Phi(x,\lambda',v))+\|\lambda-\lambda'\|\\
   &\le\kappa\|(\xi',\eta)\|+\|\eta+\zeta\|\le\kappa\sqrt{4\gamma^2+3}\|(\xi,\eta,\zeta)\|.
  \end{align*}
  Together with \eqref{aim-equa31} and \eqref{aim-ineq},
  $\widetilde{\Phi}$ is metrically subregular
  at $(\overline{x},\overline{\lambda},\overline{v})$ for the origin.

  \medskip

  Suppose that $\widetilde{\Phi}$ is metrically subregular at $(\overline{x},\overline{\lambda},\overline{v})$
  for the origin. Then there exist $\varepsilon>0$ and $\kappa>0$
  such that for all $(x,\lambda,v)\in\mathbb{B}((\overline{x},\overline{\lambda},\overline{v}),\varepsilon)$,
  \[
    {\rm dist}((x,\lambda,v),\widetilde{\Phi}^{-1}(0,0,0))
    \le\kappa{\rm dist}((0,0,0),\widetilde{\Phi}(x,\lambda,v)).
  \]
  Fix an arbitrary $(x,\lambda,v)\in\mathbb{B}((\overline{x},\overline{\lambda},\overline{v}),\varepsilon)$.
  Write $(\xi,\eta)=\Phi(x,\lambda,v)$. By the expression of $\Phi$,
  it is immediate to have that $(\xi,\eta,-\eta)\in\widetilde{\Phi}(x,\lambda,v)$.
  From the last inequality,
  \begin{align*}
   &{\rm dist}((x,\lambda,v),\Phi^{-1}(0,0))
   ={\rm dist}((x,\lambda,v),\widetilde{\Phi}^{-1}(0,0,0))\\
   &\le\kappa{\rm dist}((0,0,0),\widetilde{\Phi}(x,\lambda,v))\le\|(\xi,\eta,-\eta)\|
   \le\sqrt{2}\kappa\|\Phi(x,\lambda,v)\|
  \end{align*}
  This shows that $\Phi$ is metrically subregular at $(\overline{x},\overline{\lambda},\overline{v})$
  for the origin.
 \end{proof}

 Now applying \cite[Corollary 1]{Gfrerer16-MP} to the multifunction $\widetilde{\Phi}$,
 we have the following result.
%---------------------------------------------------------------------------------
 \begin{proposition}\label{weak-cond}
  Consider an arbitrary $(\overline{x},\overline{v})\!\in{\rm gph}\widehat{\mathcal{N}}_{\Gamma}$
  with $\mathcal{M}_{\overline{x}}(\overline{v})\!\ne\emptyset$.
  Let $\overline{\lambda}\in\!\mathcal{M}_{\overline{x}}(\overline{v})$.
  The $\Phi$ is metrically subregular at $(\overline{x},\overline{\lambda},\overline{v})$
  for the origin, if for every $0\ne(\xi,\eta,\zeta)$ with
  \begin{equation}\label{graph-deriv-equa}
   \nabla^2\langle\overline{\lambda},g\rangle(\overline{x})\xi+\!\nabla g(\overline{x})\eta+\zeta=0,
  \end{equation}
  the following implication holds:
   \begin{align}\label{equa-calm-SKKT}
   \left.\begin{matrix}
    \nabla g(\overline{x})\Delta\lambda=0,\\
    (\Delta\lambda,0)\in\mathcal{N}_{{\rm gph}\mathcal{N}_K}
    \big((g(\overline{x}),\overline{\lambda});(g'(\overline{x})\xi,\eta)\big)
    \end{matrix}\right\}
    \Longrightarrow\Delta\lambda=0.
   \end{align}
 \end{proposition}

 Since $\mathcal{N}_{{\rm gph}\mathcal{N}_K}((g(\overline{x}),\overline{\lambda});(g'(\overline{x})\xi,\eta))
 \subseteq \mathcal{N}_{{\rm gph}\mathcal{N}_K}(g(\overline{x}),\overline{\lambda})$
 for any $(\xi,\eta)\in\mathbb{X}\times\mathbb{Y}$, the implication in \eqref{equa-calm-SKKT}
 holds under the condition \eqref{coderiv-imply}, but now we can not find an example
 to illustrate that the assumption in Proposition \ref{weak-cond} is really weaker
 than the condition \eqref{coderiv-imply}. We leave this for a future
 research topic.

 \medskip

 To close this part, we take $K=\mathbb{R}_{-}$ for example to illustrate
 that there is no direct relation between the metric subregularity of $\Phi$
 and the calmness of $\mathcal{M}_{\overline{x}}$. Now, $\mathcal{M}_{\overline{x}}$
 is a polyhedral multifunction whether $g$ is polyhedral or not,
 and hence it is calm at each $\overline{v}$ for each
 $\overline{\lambda}\in\mathcal{M}_{\overline{x}}(\overline{v})$
 by \cite{Robinson81}. However, the metric subregularity of $\Phi$ depends
 on the mapping $g$. When $g$ is a linear function, clearly,
 $\Phi$ is metrically subregular at $(\overline{x},\overline{\lambda},\overline{v})$
 for the origin, but when $g$ is nonlinear, $\Phi$ does not necessarily have
 the metric subregularity at $(\overline{x},\overline{\lambda},\overline{v})$;
 for example, when $g(x)=x^2$, the mapping $\Phi$
 corresponding to the system $g(x)\in\mathbb{R}_{-}$ is not metrically
 subregular at $(\overline{x},\overline{\lambda},\overline{v})=(0,1/2,0)$.
 Indeed, by noting that
 \begin{align*}
  \Phi^{-1}(0,0)&=\Big\{(x,\lambda,v)\ |\ v=\nabla g(x)\lambda,\,g(x)={\rm min}(0,g(x)+\lambda)\Big\}\\
  &=\Big\{(x,0,0)\ |\ g(x)<0\big\}\cup\big\{(x,\lambda, v)\ |\ g(x)=0,\lambda\geq 0, v=\nabla g(x)\lambda\Big\}.
 \end{align*}
 Therefore, for any $(x,\lambda,v)\in\mathbb{R}\times\mathbb{R}\times\mathbb{R}$,
 ${\rm dist}((x,\lambda, v),\Phi^{-1}(0,0))= \|(x,v)\|$.
 Take a sequence $(x^k,\lambda^k,v^k)=(1/k,1/2,1/k)$. It is immediate to
 calculate that
 \[
   \lim_{k\to\infty}\frac{\|\Phi(x^k,\lambda^k, v^k)\|}{{\rm dist}((x^k,\lambda^k, v^k),\Phi^{-1}(0,0))}
   =\lim_{k\to\infty}\frac{\|(-v^k+2x^k\lambda^k, (x^k)^2)\|}{\|(x^k,v^k)\|}=0.
 \]
 This shows that $\Phi$ is not metrically subregular at
 $(\overline{x},\overline{\lambda},\overline{v})$ for the origin.

%-------------------------------------------------------------------------------------------
 \section{Application of graphical derivative of $\mathcal{N}_{\Gamma}$}\label{sec4}

  As an application of Theorem \ref{festimate}, we provide an exact characterization
  for the graphical derivative of the solution mapping $\mathcal{S}$ in \eqref{MSmap}
  and its isolated calmness.

 \subsection{Isolated calmness of the solution mapping $\mathcal{S}$}\label{subsec4.1}

  Firstly, we establish the relation between the graphical derivative of $\mathcal{S}$
  and that of the normal cone mapping $\mathcal{N}_\Gamma$.
  To this end, we define a map $\Psi\!:\mathbb{P}\times \mathbb{X}\rightrightarrows\mathbb{X}\times \mathbb{X}$ by
  \begin{equation}\label{Psi-map}
   \Psi(p,x):=\widetilde{F}(p,x)-{\rm gph}\mathcal{N}_\Gamma
   \ \ {\rm with}\ \widetilde{F}(p,x):=(x,-F(p,x)).
  \end{equation}
  Notice that ${\rm gph}\mathcal{S}=\widetilde{F}^{-1}({\rm gph}\mathcal{N}_\Gamma)$.
  By using Lemma \ref{Tcone-lemma}, we have the following result.
%-----------------------------------------------------------------------------
 \begin{lemma}\label{DMSmap}
  Consider an arbitrary $(\overline{p},\overline{x})\in{\rm gph}\mathcal{S}$.
  Then, the following inclusion holds
  \[
    \mathcal{T}_{{\rm gph}\mathcal{S}}(\overline{p},\overline{x})
    \!\subseteq\!\Big\{(\Delta p,\Delta x)\in\mathbb{P}\times\mathbb{X}\ |\
    \big(\Delta x,-F'((\overline{p},\overline{x});(\Delta p,\Delta x))\big)
    \in\mathcal{T}_{{\rm gph}\mathcal{N}_\Gamma}(\overline{x},-\!F(\overline{p},\overline{x}))\Big\}.
  \]
  If $\Psi$ is metrically subregular at $(\overline{p},\overline{x})$
  for the origin, the converse conclusion also holds.
 \end{lemma}

  Combining Proposition \ref{uestimate} and Theorem \ref{festimate}
  with Lemma \ref{DMSmap} and Lemma \ref{chara-icalm},
  we have the following conclusion for the isolated calmness of the solution mapping $\mathcal{S}$.
%--------------------------------------------------------------------------------------
 \begin{theorem}\label{Gderiv-MSmap}
  Consider an arbitrary $(\overline{p},\overline{x})\in{\rm gph}\mathcal{S}$
  and write $\overline{v}\!=-F(\overline{p},\overline{x})$. Suppose $\mathcal{G}$
  is metrically subregular at $\overline{x}$ for $0$. If $\mathcal{M}_{\overline{x}}$
  is isolated calm at $\overline{v}$ for some $\overline{\lambda}\in\mathcal{M}_{\overline{x}}(\overline{v})$,
  then
  \begin{equation}\label{gphS-equa}
   \!\mathcal{T}_{{\rm gph}\mathcal{S}}(\overline{p},\overline{x})
    \subseteq\!\left\{(\Delta p,\Delta x)\ \Big|
    \left.\begin{array}{ll}
     F'((\overline{p},\overline{x});(\Delta p,\Delta x))
     +\!\nabla^2\langle \overline{\lambda},g\rangle(\overline{x})\Delta x+\!\nabla g(\overline{x})\mu=0\!\\
     (g'(\overline{x})\Delta x,\mu)\!\in\!\mathcal{T}_{{\rm gph}\mathcal{N}_{K}}(g(\overline{x}),\overline{\lambda})
     \!\end{array}\right.\!\right\},
  \end{equation}
  and consequently $\mathcal{S}$ is isolated calm at $\overline{p}$ for
  $\overline{x}$ if the following implication holds:
  \begin{equation}\label{Cond-MSmap}
    -\!F'((\overline{p},\overline{x});(0,\Delta x))
    -\!\nabla^2\langle \overline{\lambda},g\rangle(\overline{x})\Delta x
    \in\nabla g(\overline{x})D\mathcal{N}_{K}(g(\overline{x})|\overline{\lambda})(g'(\overline{x})\Delta x)
    \Longrightarrow \Delta x=0.
  \end{equation}
  If, in addition, $\Phi$ is metrically subregular at $(\overline{x},\overline{\lambda},\overline{v})$
  for the origin and $\Psi$ defined in \eqref{Psi-map} is metrically subregular
  at $(\overline{p},\overline{x})$ for the origin, then the converse inclusion
  in \eqref{gphS-equa} holds and the implication \eqref{Cond-MSmap}
  is necessary for the isolated calmness of $\mathcal{S}$ at $\overline{p}$ for $\overline{x}$.
 \end{theorem}
%--------------------------------------------------------------------------------
 \begin{remark}
  When $F$ is continuously differentiable and
  $F'_{p}(\overline{p},\overline{x})\!:\mathbb{P}\to \mathbb{X}$ is surjective,
  clearly, $\Psi$ is metrically subregular at $(\overline{p},\overline{x})$ for the origin;
  when $\overline{x}$ is a nondegenerate point of the mapping $g$ w.r.t. $K$
  and the mapping $\Xi$, from Proposition \ref{property-MG} it follows that
  $\mathcal{M}_{\overline{x}}$ is isolated calm at $\overline{v}$ for some $\overline{\lambda}\in\mathcal{M}_{\overline{x}}(\overline{v})$ and
  $\Phi$ is metrically subregular at $(\overline{x},\overline{\lambda},\overline{v})$
  for the origin. Thus, Theorem \ref{Gderiv-MSmap} improves the result of
  \cite[Theorem 6.3]{Mordu15}. In particular, the isolated calmness of
  $\mathcal{S}$ at $\overline{p}$ for $\overline{x}$ does not require
  the metric subregularity of $\Phi$.
 \end{remark}

  Next we illustrate an application of Theorem \ref{Gderiv-MSmap} to
  the characterization for the isolated calmness of the KKT solution mapping
  of the canonically perturbed conic program
  \begin{equation}\label{conic-prob}
    \min_{z\in\mathbb{Z}}\big\{f(z)-\langle a,z\rangle\!: G(z)-b\in\mathcal{K}^{\circ}\big\}.
  \end{equation}
  where $p=(a,b)\in\mathbb{Z}\times\mathbb{Y}$ is the perturbation parameter,
  $f\!:\mathbb{Z}\to\mathbb{R}$ and $G\!:\mathbb{Z}\to\mathbb{Y}$ are
  twice continuously differentiable, and $\mathcal{K}\subseteq\mathbb{Y}$
  is a $C^2$-cone reducible closed convex cone.
 %----------------------------------------------------------------------------------------
 \begin{example}\label{example4.2}
  Let $f\!:\mathbb{Z}\to\mathbb{R}$ and $G\!:\mathbb{Z}\to\mathbb{Y}$ be
  twice continuously differentiable functions. Write $p:=(a,b)\in\mathbb{Z}\times\mathbb{Y}$
  and $x:=(z,\lambda)\in\mathbb{X}:=\mathbb{Z}\times\mathbb{Y}$.
  Consider the multifunction
  \[
    \mathcal{S}(p)=\Big\{x\in\mathbb{X}\ |\ 0\in F(p,x)+\widehat{\mathcal{N}}_{K}(x)\Big\}
    \ {\rm with}\
    F(p,x)=\!\left[\begin{matrix}
      \nabla\!f(z)-a +\nabla G(z)\lambda\\
      -G(z)+b
    \end{matrix}\right]
  \]
  where $K=\mathbb{Z}\times\mathcal{K}^{\circ}$.
  The $\mathcal{S}$ is exactly the KKT solution mapping associated to \eqref{conic-prob}.
  Let $\overline{p}=(0,0)$ and $\overline{x}=(\overline{z},\overline{\lambda})$
  be such that $\overline{v}=(0,G(\overline{z}))$. It is clear that
  $\mathcal{G}(x)=x-K$ is metrically subregular at
  $\overline{x}$ for the origin. Since
  \(
  \mathcal{M}_{\overline{x}}(v)
  =\big\{\mu\in\mathcal{N}_{\mathbb{X}}(\overline{z})\times\mathcal{N}_{\mathcal{K}^{\circ}}(\overline{\lambda})
  \ |\ v=\mu\big\},
 \)
 by Proposition \ref{prop-Mx} it is not hard to check that
 $\mathcal{M}_{\overline{x}}$ is isolated calm at $\overline{v}$ for $\overline{v}$.
 Now
 \[
   \Phi(x,\mu,v)=\left(\begin{matrix}
    -v+\mu\\ x-\Pi_{\mathbb{Z}\times\mathcal{K}^{\circ}}(x+\mu)
    \end{matrix}\right)\ \ {\rm and}\ \
   \widetilde{\Phi}(x,\mu,v)
    :=\left(\begin{matrix}
        -v+\mu\\ x\\ \mu\\
      \end{matrix}\right)
      -\left(\begin{matrix}
        \{0\}\\
         {\rm gph}\mathcal{N}_{K}
      \end{matrix}\right).
 \]
  Clearly, $\widetilde{\Phi}$ is metric subregular at $(\overline{x},\overline{v},\overline{v})$
  for the origin, and so is $\Phi$ at $(\overline{x},\overline{v},\overline{v})$
  for the origin by Lemma \ref{WtKKT}. In addition, since
  \(
   \widetilde{F}(p,x):=(z,\lambda,a-\!\nabla f(z)-\!\nabla G(z)\lambda, b-\!G(z))
  \)
 and $\widetilde{F}'(\overline{p},\overline{x})\!:\mathbb{X}
 \times\mathbb{X}\to\mathbb{X}\times\mathbb{X}$ is nonsingular,
 the corresponding $\Psi$ is metrically subregular at $(\overline{p},\overline{x})$
 for the origin. By Theorem \ref{Gderiv-MSmap}, $\mathcal{S}$ is isolated calm
 at $\overline{p}$ for $\overline{x}$ if and only if
 \[
   \left\{\begin{array}{ll}
     \nabla^2L(\overline{z},\overline{\lambda})\Delta z+\nabla G(\overline{z})\Delta\lambda=0,\\
     (G'(\overline{z})\Delta z,\Delta\lambda)\in\mathcal{T}_{{\rm gph}\mathcal{N}_{\mathcal{K}}}(G(\overline{z}),\overline{\lambda})
     \end{array}\right.\Longrightarrow \Delta z=0,\,\Delta\lambda=0.
  \]
  This coincides with the result in \cite[Lemma 18-19]{DingSZ17} for the perturbed
  problem \eqref{conic-prob}.
 \end{example}

 Next we use a specific example of generalized equations to illustrate Theorem \ref{Gderiv-MSmap}.
%----------------------------------------------------------------------------------------
 \begin{example}\label{example4.1}
  Consider the generalized equation \eqref{GE} with $\Gamma$ given by Example \ref{example1}
  and $F(p,x,t)=-p-(x,t)$ for $p\in\mathbb{R}^3$ and $(x,t)\in\mathbb{R}^2\times\mathbb{R}$.
  Let $\overline{p}=(0,0,0)^{\mathbb{T}}$, $(\overline{x},\overline{t})=((-1,-1)^{\mathbb{T}},0)$
  and $(\overline{\lambda},\overline{\tau})=(0_{2\times 2},-1)$.
  Since $g(\overline{x},\overline{t})=(0_{2\times 2},0)$, it is immediate to have
  \[
   \mathcal{N}_{K}(g(\overline{x},\overline{t}))=\mathbb{S}_{-}^2\times\mathbb{R}_{-}
   \ \ {\rm and}\ \ \mathcal{T}_{K}(g(\overline{x},\overline{t}))=\mathbb{S}_{+}^2\times\mathbb{R}_{+}.
  \]
  By \eqref{grad-gfun}, it is easy to verify that
  ${\rm Ker}(\nabla g(\overline{x},\overline{t}))\cap \mathcal{T}_{\mathcal{N}_{K}(g(\overline{x},\overline{t}))}(\overline{\lambda},\overline{\tau})
   =\{(0_{2\times 2},0)\}$. This shows that the SRCQ for the system $g(x,t)\in K$
  holds at $(\overline{x},\overline{t})$ w.r.t. $\overline{\lambda}$.
 However, since
 \(
  {\rm Ker}(\nabla g(\overline{x},\overline{t}))
  \cap[{\rm lin}(\mathcal{T}_{K}(g(\overline{x},\overline{t})))]^{\perp}\neq\{0_{2\times 2}\},
 \)
 it follows that $\overline{x}$ is a degenerate point.

 \medskip

 Let $(\Delta x,\Delta t)$ be such that the inclusion on the left hand side of \eqref{Cond-MSmap} holds.
 Along with the expression of $F$, there is $(\Delta\lambda,\Delta\tau)\in\mathbb{S}^2\times\mathbb{R}$
 such that $(\Delta x,\Delta t)=\nabla g(\overline{x},\overline{t})(\Delta\lambda,\Delta\tau)$
 and $(g'(\overline{x},\overline{t})(\Delta x,\Delta t),(\Delta\lambda,\Delta\tau))
 \in\mathcal{T}_{{\rm gph}\mathcal{N}_K}((g(\overline{x}),\overline{t}),(\overline{\lambda},\overline{\tau}))$.
 By \cite[Proposition 6.41]{RW98},
 \[
   \mathcal{T}_{{\rm gph}\mathcal{N}_K}((g(\overline{x}),\overline{t}),(\overline{\lambda},\overline{\tau}))
   \subseteq\mathcal{T}_{{\rm gph}\mathcal{N}_{\mathbb{S}_{+}^2}}((g_1(\overline{x},\overline{t}),\overline{\lambda})
   \times\mathcal{T}_{{\rm gph}\mathcal{N}_{\mathbb{R}_{+}}}(g_2(\overline{x},\overline{t}),\overline{\tau}).
 \]
 Together with
 \(
   g'(\overline{x},\overline{t})(\Delta x,\Delta t)
   =\left(\begin{matrix}
        {\rm Diag}(\Delta x)+\Delta t E\\
        \Delta t
        \end{matrix}\right),
 \)
 it immediately follows that
 \[
   ({\rm Diag}(\Delta x)+\Delta t E,\Delta\lambda)\in\mathcal{T}_{{\rm gph}\mathcal{N}_{\mathbb{S}_{+}^2}}((g_1(\overline{x},\overline{t}),\overline{\lambda})
   \ {\rm and}\
   (\Delta t,\Delta\tau)\in \mathcal{T}_{{\rm gph}\mathcal{N}_{\mathbb{R}_{+}}}((g_2(\overline{x},\overline{t}),\overline{\tau}).
 \]
 We calculate that $\mathcal{T}_{{\rm gph}\mathcal{N}_{\mathbb{R}_{+}}}(g_2(\overline{x},\overline{t}),\overline{\tau})
   =\{0\}\times\mathbb{R}$. Together with the last equation,
 we obtain $\Delta t=0$ and $({\rm Diag}(\Delta x),\Delta\lambda)\in\mathcal{T}_{{\rm gph}\mathcal{N}_{\mathbb{S}_{+}^2}}(g_1(\overline{x},\overline{t}),\overline{\lambda})$.
 By \cite[Corollary 3.1]{WZZhang14}, the latter implies $\mathbb{S}_{+}^2\ni{\rm Diag}(\Delta x)\,\bot\,\Delta\lambda\in\mathbb{S}_{-}^2$.
 In addition, from $(\Delta x,\Delta t)=\nabla g(\overline{x},\overline{t})(\Delta\lambda,\Delta\tau)$
 and \eqref{grad-gfun}, we have
 \(
   \Delta x={\rm diag}(\Delta\lambda).
 \)
 The two sides imply $\Delta x=0$. This shows that the implication in \eqref{Cond-MSmap} holds.
 By Theorem \ref{Gderiv-MSmap}, the mapping $\mathcal{S}$
 is isolated calm at $\overline{p}$ for $(\overline{x},\overline{t})$.
\end{example}

%----------------------------------------------------------------------------------------------
 \subsection{Estimation for (regular) coderivative of $\mathcal{N}_\Gamma$}\label{subsec4.2}

  As another application of Proposition \ref{uestimate} and
  Theorem \ref{festimate}, we provide a lower estimation for the regular coderivative of
  $\mathcal{N}_\Gamma$ and an upper estimation for the coderivative
  of $\mathcal{N}_\Gamma$, respectively, without requiring the nondegeneracy
  of the reference point.
%--------------------------------------------------------------------------------------
 \begin{proposition}\label{estimate-regNcone}
  Consider an arbitrary $(\overline{x},\overline{v})\in{\rm gph}\,\mathcal{N}_\Gamma$.
  If $\mathcal{G}$ is metrically subregular at $\overline{x}$ for $0$
  and $\mathcal{M}_{\overline{x}}$ is isolated calm at $\overline{v}$
  for some $\overline{\lambda}\in\mathcal{M}_{\overline{x}}(\overline{v})$,
  then it holds that
   \begin{align}\label{RNcone-lower}
    \widehat{\mathcal{N}}_{{\rm gph}\mathcal{N}_\Gamma}(\overline{x},\overline{v})
    &\supseteq\Big\{(\xi,\eta)\ |\ \xi\in-\nabla^2\langle\overline{\lambda},g\rangle(\overline{x})(\eta)
           +\nabla\!g(\overline{x}) \widehat{D}^*\mathcal{N}_{K}(g(\overline{x})|\overline{\lambda})(-g'(\overline{x})\eta)\Big\}.
   \end{align}
  \end{proposition}
  \begin{proof}
  Take any point $(\xi,\eta)$ from the set on the right hand side of \eqref{RNcone-lower}.
  Then, there exists $\mu\in\widehat{D}^*\mathcal{N}_{\mathcal{K}}(g(\overline{x})|\overline{\lambda})(-g'(\overline{x})\eta)$
  such that $\xi=-\nabla^2\langle\overline{\lambda},g\rangle(\overline{x})(\eta)+\nabla\!g(\overline{x})\mu$.
  To establish the inclusion \eqref{RNcone-lower}, it suffices to argue that
  $(\xi,\eta)\in[\mathcal{T}_{{\rm gph}\mathcal{N}_\Gamma}(\overline{x},\overline{v})]^{\circ}
  =\widehat{\mathcal{N}}_{{\rm gph}\mathcal{N}_\Gamma}(\overline{x},\overline{v})$.
  Let $(d,w)$ be an arbitrary point from $\mathcal{T}_{{\rm gph}\mathcal{N}_\Gamma}(\overline{x},\overline{v})$.
  By Theorem \ref{uestimate}, there is $u\in\mathbb{Y}$ such that
  \[
    (g'(\overline{x})d,u)\in\mathcal{T}_{{\rm gph}\mathcal{N}_K}(g(\overline{x}),\overline{\lambda})
    \ \ {\rm and}\ \
    w=\nabla^2\langle\overline{\lambda},g\rangle(\overline{x})(d)+\nabla\!g(\overline{x})u.
  \]
  Together with $\xi=-\nabla^2\langle\overline{\lambda},g\rangle(\overline{x})(\eta)+\nabla\!g(\overline{x})\mu$,
  it follows that
  \begin{align}\label{temp-equa40}
   \langle (\xi,\eta),(d,w)\rangle
   &=\langle d,-\nabla^2\langle\overline{\lambda},g\rangle(\overline{x})(\eta)+\nabla\!g(\overline{x})\mu\rangle
    +\langle \nabla^2\langle\overline{\lambda},g\rangle(\overline{x})(d)+\nabla\!g(\overline{x})u,\eta\rangle\nonumber\\
   &=\langle d,\nabla\!g(\overline{x})\mu\rangle+\langle\nabla\!g(\overline{x})u,\eta\rangle.
  \end{align}
  Notice that $\mu\in\widehat{D}^*\mathcal{N}_{K}(g(\overline{x})|\lambda)(-g'(\overline{x})\eta)$.
  Hence, $(\mu,g'(\overline{x})\eta)\in\widehat{\mathcal{N}}_{{\rm gph}\mathcal{N}_K}(g(\overline{x}),\lambda)$.
  Since $(g'(\overline{x})d,u)\in\mathcal{T}_{{\rm gph}\mathcal{N}_K}(g(\overline{x}),\lambda)$
 and $\widehat{\mathcal{N}}_{{\rm gph}\mathcal{N}_{K}}(g(\overline{x}),\lambda)
  =[\mathcal{T}_{{\rm gph}\mathcal{N}_{K}}(g(\overline{x}),\lambda)]^{\circ}$,
  it holds that
  \[
    \langle(\mu,g'(\overline{x})\eta),(g'(\overline{x})d,u)\rangle
    = \langle d,\nabla\!g(\overline{x})\mu\rangle+\langle\nabla\!g(\overline{x})u,\eta\rangle
    \le 0.
  \]
  Together with \eqref{temp-equa40}, $\langle (\xi,\eta), (d,w)\rangle\leq 0$.
  Thus, we obtain $(\xi,\eta)\in[\mathcal{T}_{{\rm gph}\mathcal{N}_\Gamma}(\overline{x},\overline{v})]^{\circ}$.
 \end{proof}
%--------------------------------------------------------------------------------------
 \begin{proposition}\label{estimate-Ncone}
  Consider an arbitrary $(\overline{x},\overline{v})\in{\rm gph}\,\mathcal{N}_\Gamma$.
  If $\mathcal{M}_{\overline{x}}$ is isolated calm at $\overline{v}$ for 
  $\overline{\lambda}\in\mathcal{M}_{\overline{x}}(\overline{v})$ and
  $\Phi$ is metrically subregular at $(\overline{x},\overline{\lambda},\overline{v})$
  for the origin, then it holds that
   \begin{equation}\label{Ncone-upper}
    \mathcal{N}_{{\rm gph}\mathcal{N}_\Gamma}(\overline{x},\overline{v})
    \subseteq\Big\{(\xi,\eta)\ |\ \xi\in-\nabla^2\langle\overline{\lambda},g\rangle(\overline{x})(\eta)
           +\nabla\!g(\overline{x}) D^*\mathcal{N}_{K}(g(\overline{x})|\overline{\lambda})(-g'(\overline{x})\eta)\Big\}.
   \end{equation}
  \end{proposition}
  \begin{proof}
  Let $\mathcal{A}$ be the linear mapping appearing in the proof of Lemma \ref{lestimate},
  and let $\mathcal{U}$ be an arbitrary neighborhood of $(\overline{x},\overline{v})$.
  We first argue that $\mathcal{A}^{-1}(\mathcal{U})\cap\Phi^{-1}(0,0)$ is bounded.
  If not, by the definition of $\mathcal{A}$, there exist $\{(x^k,v^k)\}$ converging to
  $(\overline{x},\overline{v})$ and an unbounded $\{\mu^k\}$ such that for each $k$, $(x^k,\mu^k,v^k)\in\mathcal{A}^{-1}(\mathcal{U})\cap\Phi^{-1}(0,0)$, that is,
  \[
    v^k=\nabla g(x^k)\mu^k\ \ {\rm and}\ \ \mu^k\in\mathcal{N}_K(g(x^k))
    \quad\ \forall k.
  \]
  Write $\widetilde{\mu}^k=\frac{\mu^k}{\|\mu^k\|}$ and $\widetilde{v}^k=\frac{v^k}{\|\mu^k\|}$.
  We assume (if necessary taking a subsequence) that $\widetilde{\mu}^k\to\widetilde{\mu}$
  with $\|\widetilde{\mu}\|=1$. Notice that $\widetilde{v}^k=\nabla g(x^k)\widetilde{\mu}^k$ and $\widetilde{\mu}^k\in\mathcal{N}_K(g(x^k))$. From the outer semicontinuity of $\mathcal{N}_K$,
  $\widetilde{\mu}\in\mathcal{N}_K(g(\overline{x}))$ and $\nabla g(\overline{x})\widetilde{\mu}=0$.
  This, by the isolated calmness of $\mathcal{M}_{\overline{x}}$ at $\overline{v}$,
  implies that $\widetilde{\mu}=0$, a contradiction to $\|\widetilde{\mu}\|=1$.
  Now, by \cite[Theorem 6.43]{RW98}, from ${\rm gph}\mathcal{N}_\Gamma\cap(\mathcal{V}\times\mathbb{X})
  =\mathcal{A}(\Phi^{-1}(0,0))\cap(\mathcal{V}\times\mathbb{X})$ for a neighborhood $\mathcal{V}$
  of $\overline{x}$ we have
  \begin{align}\label{temp-Ncone-Nomega}
   \mathcal{N}_{{\rm gph}\mathcal{N}_\Gamma}(\overline{x},\overline{v})
   &\subseteq\bigcup_{\overline{z}\in\mathcal{A}^{-1}(\overline{x},\overline{v})\cap\Phi^{-1}(0,0)}
    \Big\{(\xi,\eta)\in\mathbb{X}\times\mathbb{Y}\ |\ \mathcal{A}^*(\xi,\eta)\in\mathcal{N}_{\Phi^{-1}(0,0)}(\overline{z})\Big\}\nonumber\\
    &=\Big\{(\xi,\eta)\in\mathbb{X}\times\mathbb{Y}\ |\ \mathcal{A}^*(\xi,\eta)\in\mathcal{N}_{\Phi^{-1}(0,0)}(\overline{x},\overline{\lambda},\overline{v})\Big\},
  \end{align}
  where the equality is by the definition of $\mathcal{A}$
  and $\mathcal{M}_{\overline{x}}(\overline{v})=\{\overline{\lambda}\}$.
  Since $\Phi$ is metrically subregular at $(\overline{x},\overline{\lambda},\overline{v})$
  for the origin, applying the first part of Lemma \ref{normal-cone-lemma} yields
  \[
    \mathcal{N}_{\Phi^{-1}(0,0)}(\overline{x},\overline{\lambda},\overline{v})
    \subseteq \bigcup_{(d,w)\in\mathbb{X}\times\mathbb{Y}}D^*\Phi(\overline{x},\overline{\lambda},\overline{v})(d,w).
  \]
  For any given $(d,w)\in\mathbb{X}\times\mathbb{Y}$,
  from the proof of Proposition \ref{property-MG}(a) it follows that
  \[
    D^*\Phi(\overline{x},\overline{\lambda},\overline{v})(d,w)
    \!=\!\left[\begin{matrix}
           \nabla^2\langle\overline{\lambda},g\rangle(\overline{x})d
           \!+\!\nabla\!g(\overline{x})w\\
            g'(\overline{x})d\\
            -d
       \end{matrix}\right]
       +\left(\begin{matrix}
      \nabla\!g(\overline{x})\\
       I\\ 0
       \end{matrix}\right)\!D^*\Pi_K(g(\overline{x})\!+\!\overline{\lambda})(-w).
  \]
  By combining the last two equations with \eqref{Proj-normal}, it is not difficult to obtain that
  \begin{align*}
   &\mathcal{N}_{\Phi^{-1}(0,0)}(\overline{x},\overline{\lambda},\overline{v})\\
   &\subseteq\!\bigcup_{(d,w)\in\mathbb{X}\times\mathbb{Y}}
    \left\{\!\left(\begin{matrix}
            \nabla^2\langle\overline{\lambda},g\rangle(\overline{x})d+\!\nabla g(\overline{x})u'\\
            g'(\overline{x})d+u'-w\\ -d
            \end{matrix}\right)\ \bigg|\
    \left(\begin{matrix}
     u'\\u'\!-w
     \end{matrix}\right)\in \mathcal{N}_{{\rm gph}\mathcal{N}_K}(g(\overline{x}),\overline{\lambda})\right\}.
  \end{align*}
  Together with \eqref{temp-Ncone-Nomega} and $\mathcal{A}^*(\xi,\eta)=(\xi,0,\eta)$,
  we obtain the following inclusion
  \[
    \mathcal{N}_{{\rm gph}\mathcal{N}_\Gamma}(\overline{x},\overline{v})
    \subseteq\left\{(\xi,\eta)\in\mathbb{X}\times\mathbb{Y}\ \Big|
    \left.\begin{array}{ll}
     \xi=\nabla^2\langle\overline{\lambda},g\rangle(\overline{x})(-\eta)+\nabla\!g(\overline{x})z,\\
     (z,g'(\overline{x})\eta)\!\in\mathcal{N}_{{\rm gph}\mathcal{N}_{K}}(g(\overline{x}),\overline{\lambda})
     \end{array}\right.\!\right\},
   \]
  which is equivalent to the inclusion in \eqref{Ncone-upper}.
  The proof is then completed.
  \end{proof}

  Exact characterizations for $\widehat{\mathcal{N}}_{{\rm gph}\mathcal{N}_\Gamma}(\overline{x},\overline{v})$
  and $\mathcal{N}_{{\rm gph}\mathcal{N}_\Gamma}(\overline{x},\overline{v})$
  were given in \cite[Theorem 4.1]{Mordu15} and in \cite[Theorem 7]{Outrata11},
  respectively, under the standard reducibility and nondegeneracy assumption,
  and they were recently obtained in \cite{Gfrerer17} under a weakened version
  of reducibility but still the nondegeneracy assumption. Here, we only provide
  a one-sided estimation without the nondegeneracy assumption, and it is not unclear
  whether the converse inclusions in \eqref{RNcone-lower}-\eqref{Ncone-upper}
  hold or not without nondegeneracy.

  \bigskip
  \noindent
  {\large\bf Acknowledgements}\ \ The authors are deeply grateful for
  the two referees' comments, which give them much help to improve
  the original manuscript. The authors also would like to thank Professor Mordukhovich,
  from Wayne State University, for his helpful suggestions on the revision of this manuscript.
%---------------------------------------------------------------------------------------------------

 \medskip
 \noindent
 {\bf\large Appendix}
%---------------------------------------------------------------------------------
 \begin{alemma}\label{calculus-rule}
  Let $G_1\!:\mathbb{X}\to\mathbb{Y}$ be a single-valued mapping and
  $G_2:\mathbb{X}\rightrightarrows\mathbb{Z}$ be an arbitrary set-valued mapping.
  Define the set-valued mapping $G\!:\mathbb{X}\rightrightarrows\mathbb{Y}\times\mathbb{Z}$ by
  \(
    G(x):=\left(\begin{matrix}
               G_1(x)\\ G_2(x)
              \end{matrix}\right)
  \)
  for $x\in\mathbb{X}$. Consider a point $(\overline{x},(\overline{y},\overline{z}))\in{\rm gph}\,G$.
  If $G_1$ is strictly differentiable at $\overline{x}$, then
  \[
    D^*G(\overline{x}|(\overline{y},\overline{z}))(\Delta u,\Delta v)
    =\nabla\!G_1(\overline{x})\Delta u + D^*G_2(\overline{x}|\overline{z})(\Delta v)
    \quad \forall(\Delta u,\Delta v)\in\mathbb{Y}\times\mathbb{Z}
  \]
 \end{alemma}
 \begin{proof}
  Let $H(x):=\left(\begin{matrix}
              0\\ G_2(x)
             \end{matrix}\right)$ for $x\in\mathbb{X}$.
 By the definition of coderivative, we have that
 \[
  D^*H(\overline{x}|(0,\overline{z}))(\Delta u,\Delta v)=D^*G_2(\overline{x}|\overline{z})(\Delta v).
 \]
  Notice that $G(x)=H(x)+
  \left(\begin{matrix}
              G_1(x)\\ 0
             \end{matrix}\right)$ for $x\in\mathbb{X}$.
 By \cite[Theorem 1.62]{Mordu06}, it follows that
  \[
    D^*G(\overline{x}|(\overline{y},\overline{z}))(\Delta u,\Delta v)
    =\nabla\!G_1(\overline{x})\Delta u +D^*H(\overline{x}|(0,\overline{z}))(\Delta u,\Delta v).
  \]
  The desired result then follows by combining the last two equations.
 \end{proof}
 \end{document}